\documentclass[12]{article}
\usepackage[english]{babel}
\usepackage{amsmath}
\usepackage{amsthm}
\usepackage{amssymb}
\usepackage{amscd}
\usepackage[latin1]{inputenc}
\usepackage[all]{xy}
\usepackage{color}
\newtheorem{theorem}{Theorem}[section]
\newtheorem{definition}[theorem]{Definition}
\newtheorem{lemma}[theorem]{Lemma}
\newtheorem{proposition}[theorem]{Proposition}

\newtheorem{corollary}[theorem]{Corollary}
\newtheorem{remark}[theorem]{Remark}

\newtheorem{example}[theorem]{Example}
\newtheorem{examples}[theorem]{Examples}
\newtheorem{question}[theorem]{Question}

\newcommand{\A}{\mathcal{A}}
\newcommand{\T}{\mathcal{T}}
\newcommand{\mono}{\hookrightarrow}
\newcommand{\epi}{\twoheadrightarrow}
\newcommand{\F}{\mathcal{F}}

\newcommand{\G}{\mathcal{G}}
\newcommand{\Ha}{\mathcal{H}}
\newcommand{\D}{\mathcal{D}}
\newcommand{\W}{\mathcal{W}}

\newcommand{\U}{\mathcal{U}}

\newcommand{\Ht}{\mathcal{H}_{\mathbf{t}}}
\newcommand{\Gen}{\text{Gen}}
\newcommand{\Cogen}{\text{Cogen}}
\newcommand{\Pres}{\text{Pres}}
\newcommand{\Copres}{\text{Copres}}
\newcommand{\Hom}{\text{Hom}}
\newcommand{\Ext}{\text{Ext}}
\newcommand{\Ker}{\text{Ker}}
\newcommand{\Coker}{\text{Coker}}
\newcommand{\Add}{\text{Add}}
\newcommand{\Produc}{\text{Prod}}
\newcommand{\add}{\text{add}}

\newcommand{\keywords}[1]{\noindent\textbf{Key words and phrases: } #1}
\newcommand{\MSC}[1]{\noindent\textbf{2010 Mathematics Subject Classification. } #1}

\title{The HRS tilting process and Grothendieck hearts of t-structures}

\author{Carlos E. Parra \footnote{The first named author was supported by CONICYT/FONDECYT/Iniciaci\'on/11160078} \and Manuel Saor\'{\i}n \footnote{The second named author was supported by the research projects from Spanish Ministerio de Econom\'{\i}a y Competitividad (MTM2016-77445-P) and from the Fundaci\'on `S\'eneca' of
Murcia (19880/GERM/15), with a part of FEDER funds.}\footnote{Both authors warmly thank Simone Virili for his comments, suggestions and help during the preparation of the paper.}}

\begin{document}

\maketitle

\begin{abstract}

In this paper we revisit the problem of determining when the heart of a t-structure is a Grothendieck category, with special attention to the case of the Happel-Reiten-Smal\o \ (HSR) t-structure in the derived category of a Grothendieck category associated to a torsion pair in the latter. We revisit the HRS tilting process deriving from it a lot of information on the HRS t-structures which have a projective generator or an injective cogenerator, and obtain several bijections between classes of pairs $(\A,\mathbf{t})$ consisting of an abelian category and a torsion pair in it. We use these bijections to re-prove, by different methods, a recent result of Tilting Theory and the fact that if $\mathbf{t}=(\T,\F)$ is a torsion pair in a Grothendieck category $\G$, then the heart of the associated HRS t-structure is itself a Grothendieck category if, and only if,  $\mathbf{t}$ is of finite type. We survey this last problem and recent results after its solution. 

\end{abstract}

\keywords derived category, Grothendieck category, Happel-Reiten-Smal\o \ t-structure, heart of a t-structure, torsion pair, t-structure. \\
\MSC  18E15, 18E30, 18E40, 16B50, 16E30, 16E35.

\section{Introduction}

The aim of this paper is twofold. On one side we want to give a summary of the main results related with the following question:

\begin{question}\label{Question Groth.}
When is the heart of a t-structure a Grothendieck category?
\end{question}

We shall mainly concentrate in the route leading to the answer to the question in the case when the ambient triangulated category is the (unbounded) derived category $\D(\G)$ of a Grothendieck category $\G$ and the t-structure is the Happel-Reiten-Smal\o \ (HRS) t-structure in $\D(\G)$ associated to a torsion pair in $\G$ (see Example \ref{examples-t-structures}(2)). But we include a final short section, where we briefly summarize the main results for general triangulated categories with coproducts and arbitrary t-structures. As a second goal, we want to revisit the HRS tilting process and show that it allows to prove in an easy way parts of recent results in the literature, and that, with the help of a recent approach to the problem using purity, one can re-prove the answer to Question \ref{ques.HRS-tstructure} below by methods completely different to those used to get the earlier answer.

All throughout the paper, unless otherwise stated, all categories will be additive. We will  mainly use two types of  
 categories commonly studied in \emph{Homological Algebra}, concretely  \emph{abelian categories} and $\emph{triangulated categories}$ (we refer to  \cite{S} and \cite{N}  for the respective definitions). The key concept for us is that of a \emph{t-structure} in a triangulated category, introduced by Beilinson, Bernstein and Deligne \cite{BBD} in their treatment of perverse sheaves. Roughly speaking a t-structure in the triangulated category $\D$ is a pair $\tau =(\mathcal{U},\mathcal{W})$ of  full subcategories satisfying some axioms (see Definition \ref{Def.t-structure} for the details) which guarantee that the intersection $\mathcal{H}_\tau=\mathcal{U}\cap\mathcal{W}$ is an abelian category, commonly called the \emph{heart} of the t-structure.   This abelian category  comes with a cohomological functor $H_\tau^0:\D\longrightarrow\mathcal{H}_\tau$. In \cite{BBD} the category of perverses sheaves on a variety $\mathbb{X}$ appeared as the heart of a t-structure in $\D^b(\mathbb{X})$, the bounded derived category of coherent sheaves on $\mathbb{X}$. 

In several modern developments of Mathematics, as  Motive Theory, the homological approach to Mirror Symmetry, Modular Representation of finite groups, Representation Theory of Algebras, among others, the role of t-structures is fundamental.  For this reason it is important to know when the heart of a t-structure has nice properties as an abelian category. Vaguely speaking, one would ask:  When is the heart of a given t-structure a nice category?. Trying to make sense of the adjective `nice' here, one commonly uses the following ``hierarchy" among abelian categories introduced by Grothendieck \cite{G}. We say that an abelian category $\A$ is:

\begin{enumerate}
\item AB3 (resp. AB3*) when it has (arbitrary set-indexed) coproducts (resp. products);

\item AB4 (resp. AB4*) when it is AB3 (resp. AB3*) and the coproduct functor $\coprod: [\Lambda,\A] \rightarrow \A$ (resp. product functor $\prod: [\Lambda,\A] \rightarrow \A$) is exact, for each set $\Lambda$; 
\item  AB5 (resp. AB5*) when it is AB3 (resp. AB3*) and the direct limit functor $\varinjlim: [\Lambda, \A] \rightarrow \A$ (resp. inverse limit functor $\varprojlim:[\Lambda^{\text{op}},\A] \rightarrow \A$) is exact, for each directed set $\Lambda.$
\item a Grothendieck category, when it is AB5 and has a generator or, equivalently, a set of generators. 
\end{enumerate}

 Grothendieck categories appear quite naturally in Algebra and Geometry and their behavior is, in many aspects, similar to that of module categories over a ring (see \cite[Chapter V]{S}). For instance, such a category has enough injectives and every object in it has an injective envelope. Even more, by a famous theorem of Gabriel and Popescu (see \cite{GP}, and also \cite[Theorem X.4.1]{S}), such a category is always a Gabriel localisation of a module category, which roughly means that it is obtained from such a category by formally inverting some morphisms.  This is the main reason why the study of when the heart of a t-structure is a Grothendieck category, i.e. Question \ref{Question Groth.},  has deserved most of the attention, apart of the study of  when it is a module category, that we barely touch in this paper. When one starts approaching the question, one quickly sees that it is hopeless unless some extra hypotheses are imposed on the ambient triangulated category $\D$ and/or on the t-structure $\tau$ itself. For instance, it is unavoidable to require that $\D$ has coproducts or, at least, to guarantee that  coproducts in $\D$ of objects in the heart of $\tau$ always exist. On the other hand, the problem gets quite complicated if the coproduct in $\mathcal{H}_\tau$ and the coproduct in $\D$  of a given family of objects in $\mathcal{H}_\tau$ do not coincide. A way of ensuring that they coincide is to require that the t-structure be \emph{smashing}, i.e. that the co-aisle $\mathcal{W}$ of the t-structure is closed under coproducts in $\D$. Therefore, instead of the initial question, the following one has more hopes of being answered and has deserved a lot of attention in recent times (see Section \ref{sec.beyondHRS}):
 
 \begin{question} \label{ques.smashing-t-structure}
 Let $\D$ be a triangulated category with coproducts and let $\tau =(\mathcal{U},\mathcal{W})$ be a smashing t-structrure in $\D$. When is the heart of $\tau$ a Grothendieck category?
 \end{question}

 Although studied historically first, the question for the HRS t-structure is a particular case of this last question. Namely, if $\G$ is a Grothendieck category, then its derived category $\D(\G)$ is the prototypical example of a triangulated category with coproducts (and also products). When a torsion pair $\mathbf{t}=(\T,\F)$ is given in $\G$, the associated HRS t-structure in $\D(\G)$ is smashing. So restricted to this particular example, the last question is re-read as follows, and it is the main problem that we survey and  re-visit in this paper:
 
 \begin{question} \label{ques.HRS-tstructure}
 Let $\G$ be a Grothendieck category, let $\mathbf{t}=(\T,\F)$ be a torsion pair in $\G$ and let $\Ht$ be the heart of the associated HRS t-structure in $\D(\G)$. When is $\Ht$ a Grothendieck category?
 \end{question}

 Let's now have a look at the new results and/or proofs of the paper. On what concerns our new look at the HRS tilting process, we give a series of results leading to a list of bijections between pairs consisting of a class of abelian categories and a class of torsion pairs in them (see Corollaries \ref{cor.bijections-from-HRStilt} and \ref{cor.AB5-via-pureinjectivity} for the complete list). We just point out in this introduction two of those results (see Theorem \ref{thm.tilting theorem} and Corollary \ref{cor.module-category}). The first one identifies the torsion pairs with cogenerating torsion class for which the heart is AB3 and has a projective generator (see also Theorem \ref{thm.cotilting theorem} for its dual).
 
 \begin{theorem} \label{thm.tilting theorem-intro}
Let $\A$ be an abelian category and let $\mathbf{t}=(\mathcal{T},\mathcal{F})$ be a torsion pair in $\A$. The following assertions are equivalent:
\begin{enumerate}
\item $\mathbf{t}$ is a tilting torsion pair.
\item $\mathbf{t}$ is a co-faithful torsion pair whose heart $\Ht$ is an AB3 abelian category with a projective generator. 
\end{enumerate} 
In this case,  $V$ is  1-tilting object such that $\T=\Gen(V)$ if,  and only if,  $V[0]$ is a projective generator of $\Ht$. Moreover, an object $P$ of $\Ht$ 
is a projective generator of this latter category if, and only if, it is isomorphic to $V[0]$ for some 1-tilting object $V$ of $\A$ such that $\T=\Gen (V)$. 
\end{theorem}

The next one characterizes when a co-faithful torsion pair has a heart which is a module category (see  Corollary \ref{cor.module-category}).

\begin{corollary} \label{cor.module-category-intro}
Let $\A$ be an abelian category and let $\mathbf{t}=(\T,\F)$ be a torsion pair in $\A$. The following assertions are equivalent:

\begin{enumerate}
\item There is a classical 1-tilting set $\T_0$ (resp. a classical 1-tilting object $V$) such that $\T=\Gen(\T_0)$ (resp. $\T=\Gen(V)$).
\item $\mathbf{t}$ is a co-faithful torsion pair whose heart $\Ht$ is equivalent to the module category over a small pre-additive category (resp. over a ring). 
\end{enumerate}
\end{corollary}

One of the consequences of our new visit to the HRS tilting process is a new easy proof of the $n=1$ case of Positselsky-Stovicek tilting-cotilting correspondence (see Corollary \ref{cor.Po-St-tilting-cotilting-corresp}):

\begin{corollary}[Positselski-Stovicek] \label{cor.Po-St-tilting-cotilting-corresp-intro}
The HRS tilting process gives a one-to-one correspondence between:
\begin{enumerate}
\item The pairs $(\A ,\mathbf{t})$ consisting of an AB3*  abelian category $\A$ with an injective cogenerator and a tilting torsion pair $\mathbf{t}$ in $\A$;
\item The pairs $(\mathcal{B},\bar{\mathbf{t}})$ consisting of an AB3  abelian category $\mathcal{B}$ with a projective generator and a cotilting torsion pair $\bar{\mathbf{t}}$ in $\mathcal{B}$.
\end{enumerate}
\end{corollary}
 
 On what concerns Question \ref{ques.HRS-tstructure},  we will re-prove,  by completely different methods (see Theorem \ref{thm.Parra-Saorin-Theorem}),  the 'only if' implication, which is the hardest one,  of the following earlier result:
 
 \begin{theorem} (\cite[Theorem 1.2]{PS2}]) \label{thm.PS}
  Let $\G$ be a Grothendieck category, let $\mathbf{t}=(\T,\F)$ be a torsion pair in $\G$ and let $\Ht$ be the heart of the associated HRS t-structure in $\D(\G)$. Then $\Ht$ is a Grothendieck category if, and only if, the torsionfree class $\mathcal{F}$ is closed under taking direct limits in $\mathcal{G}$. 
 \end{theorem}
 
 The reader is referred to Section 5 for a summary of recent results concerning Question \ref{ques.smashing-t-structure}, that go beyond the HRS situation.

 The organization of the paper goes as follows. In Section 2 we introduce the main concepts needed for the understanding of the paper,  specially torsion pairs in abelian categories and t-structures in triangulated categories, with a look also at the HRS tilting process. It turns out that if one starts with a pair $(\A,\mathbf{t})$ consisting of an abelian category $\A$ and a torsion pair $\mathbf{t}$, then the new abelian category obtained by tilting $\A$ with respect to $\mathbf{t}$ need not have $Hom$ sets. Corollary \ref{cor.heart-genuine-category} gives the precise conditions to get $Hom$ sets. In Section 3 we study when the heart of (the HRS t-structure associated to) a torsion pair in an abelian category has either a projective generator or an injective cogenerator. This leads naturally to quasi-(co)tilting torsion pairs (see Proposition \ref{prop.projective-generator-implies-qtilting}) in abelian categories. Then expanded versions of Theorem  \ref{thm.tilting theorem-intro} and its dual are proved (see Theorems \ref{thm.tilting theorem} and \ref{thm.cotilting theorem}). As a particular case, we then give an expanded version Corollary \ref{cor.module-category-intro} (see Corollary \ref{cor.module-category}) that characterizes when the heart of a co-faithful torsion pair is a module category.
 We end the section by giving a series of bijections induced by the HRS tilting process (see Proposition \ref{cor.bijections-from-HRStilt}), from which Corollary \ref{cor.Po-St-tilting-cotilting-corresp-intro} is deduced. In Section 4 we study Question \ref{ques.HRS-tstructure}. Subsections 4.1 and 4.2 are dedicated to show the milestones of the route that led to the solution of the problem in \cite{PS2}, i.e. to the proof of Theorem \ref{thm.PS}. Subsection 4.3 briefly summarizes recent results by Bazzoni, Herzog, Prihoda, Saroch and Trlifaj about the same question in the particular case when the torsion pair is tilting. We end the section by re-proving, using a recent characterization of the AB5 condition by Positselski and Stovicek (see \cite{Po-St}), the fact that if the heart of a torsion pair in a Grothendieck category is itself a Grothendieck category, then the torsion pair is of finite type. The final Section 5 shows the most recent results and the present state of Question \ref{ques.smashing-t-structure}.


\vspace{0,3 cm}

\section{Preliminaries}
In the rest of the paper, whenever $\mathcal{C}$ is an additive category 
and $\mathcal{X}$ is any class of objects, we shall denote by 
$\mathcal{X}^\perp$ (resp. $_{}^\perp\mathcal{X}$) the full subcategory 
consisting of the objects $Y$ such that $\mathcal{C}(X,Y)=0$ (resp. 
$\mathcal{C}(Y,X)=0$), for all $X\in\mathcal{X}$. When $\mathcal{X}=\{X\}$ for simplicity we will write $X^{\perp}$ (resp. $_{}^{\perp}X$) instead of $\mathcal{X}^{\perp}$ (resp. $_{}^{\perp}\mathcal{X}$).

Unless explicitly said otherwise, in the rest of the paper the letter $\mathcal{A}$ will denote an abelian category.

\subsection{Torsion pairs}
\begin{definition}\label{def. torsion pair} \normalfont
A \emph{torsion pair} in $\A$ is a pair $\mathbf{t}=(\T,\F)$ of full subcategories satisfying the following two conditions:
\begin{enumerate}
\item [1)] $\A(T,F)=0$, for all $T\in\T$ and $F\in\F$;
\item [2)] For each object $X$ of $\A$, there is an exact sequence
$$\xymatrix{0 \ar[r] & T_X \ar[r] & X \ar[r] & F_X \ar[r] & 0}$$
\end{enumerate}
where $T_X \in \T$ and $F_X\in \F$. 
\end{definition}

A \emph{torsion class} in $\A$ is a class of objects $\T$  that appears as first component of a torsion pair in $\A$. A \emph{torsionfree class} $\F$ is defined dually. Note that in a torsion pair we have $\F=\T^\perp$ 
and $\T= \hspace{0,01 cm} ^\perp\F$. On the other hand, in the sequence above $T_X$ and $F_X$ depend functorially on $X$, so that the assignment $X \leadsto T_X$ (resp. $X \leadsto F_X$) underlies a functor $t:\A \rightarrow \T$ (resp. $(1:t): \A \rightarrow \F$), which is right (resp. left) adjoint of the inclusion functor $\iota_\T:\T \mono \A$ (resp. $\iota_\F:\F \mono \A$). The composition $\iota_\T \circ t: \A \rightarrow \A$ (resp. $\iota_\F \circ (1:t): \A \rightarrow \A$), which we will still denote by $t$ (resp. $(1:t)$),  is called the \emph{torsion radical} (resp. \emph{torsion coradical})  associated to $\mathbf{t}$.   
\vspace{0.3 cm}

In particular situations, torsion and torsionfree classes are identified by the satisfaction of some closure properties. Recall that an abelian category is called \emph{locally small} when the subobjects of any given object form a set.

\begin{proposition}\label{prop. torsion class=ecq}
Let $\A$ be an abelian category  and let $\T\subseteq \A$ (resp. $\F\subseteq \A$) be a full subcategory. Consider the following assertions:
\begin{enumerate}
 \item  $\T$ (resp. $\F$) is a torsion (resp. torsionfree) class; 
  \item$\T$ (resp. $\F$) is closed under taking quotients (=epimorphic images)  (resp. subobjects), extensions and coproducts (resp. products), when these exist in $\A$. 
  \end{enumerate}
  The implication $(1)\Longrightarrow (2)$ holds true. When $\A$ is 
   AB3 (resp. AB3*) and   locally small, also $(2)\Longrightarrow (1)$ holds.  
\end{proposition}
\begin{proof}
The implication $(1)\Longrightarrow (2)$ follows from the equalities $\F=\T^\perp$ 
and $\T= \hspace{0,01 cm} ^\perp\F$. For $(2)\Longrightarrow (1)$ see
\cite[Proposition VI.2.1]{S}. 
\end{proof}

Recall that, in any category $\mathcal{C}$, a class of objects $\mathcal{X}$ is called a \emph{generating (resp. cogenerating) class} when, for each object $C\in\text{Ob}(\mathcal{C})$, there is an epimorphism $X_C\twoheadrightarrow C$ (resp. monomorphism $C\rightarrowtail X_C$), for some object $X_C\in\mathcal{X}$. 

We recall some particular cases of torsion pairs:

\begin{definition} \normalfont
Let $\mathbf{t}=(\T,\F)$ be a torsion pair in $\A$. We will say that $\mathbf{t}$ is:

\begin{enumerate}
\item \emph{faithful} (resp. \emph{co-faithful}) when $\F$ (resp. $\T$) is a generating (resp. cogenerating) class of $\A.$

\item \emph{of finite type} when direct limits in $\A$ of objects in $\F$ exist and are in $\F$.
\end{enumerate}
\end{definition}

\begin{remark}\normalfont
In \cite{HRS} faithful (resp. co-faithful) torsion pairs are called cotilting (resp. tilting). In this paper we separate from that terminology, reserving the term 'cotilting' (resp. 'tilting') for torsion pairs defined by 1-cotilting (resp. 1-tilting) objects (see Definition \ref{def. tilting object}).
\end{remark}




\subsection{t-Structures}
In the sequel, the letter $\D$ will  denote a triangulated category and $?[1]:\D \rightarrow \D$ will be its suspension functor. Moreover, we put $?[0]=1_{\D}$ and $?[k]$ the $k$-th power of $?[1]$, for each integer $k$. We will denote the triangles on $\D$ by $\xymatrix@C=1,3pc{X \ar[r] & Y \ar[r] & Z \ar[r]^{+} & }$, or also $\xymatrix@C=1,6pc{X \ar[r] & Y \ar[r] & Z \ar[r] & X[1]}$. 

\begin{definition} \label{Def.t-structure} \normalfont
Let $(\D,?[1])$ be a triangulated category. A \emph{t-structure} on $\D$ is a couple of full subcategories closed under direct summands $(\U,\W)$ such that:
\begin{enumerate}
\item $\U[1]\subseteq \U;$
\item $\D(U,W[-1])=0$, for all $U\in \U$ and $W \in \W$;
\item For each $X\in \D$, there is a distinguished triangle:
$$U_X \rightarrow X \rightarrow V_X \rightarrow U_X[1]$$
with $U_X\in \U$ and $V_X\in \W[-1]$.
\end{enumerate}
\end{definition}
In such case, the subcategory $\U$ is called the \emph{aisle} of the t-structure, and $\W$ is called the \emph{coaisle}. Note that in such case, we have $\W[-1]=\U^{\perp}$ and $\U=\ ^{\perp}(\W[-1])= \ ^{\perp}(\U^{\perp}).$ For this reason, we will write the t-structures using the following notation $(\U,\U^{\perp}[1]).$ On the other hand, the objects $U_X$ and $V_X$ in the previous triangle are uniquely determined by $X$, up to isomorphism, so that the assignment $X \leadsto U_X$ (resp. $X \leadsto V_X$) underlies a functor $\tau_{\U}^{\leq}:\D \rightarrow \U$ (resp. $\tau_{\U}^{>}: \D \rightarrow \U^{\perp}$) which is right (resp. left) adjoint of the inclusion functor $\iota_{\U}:\U \mono \D$ (resp. $\iota_{\U^{\perp}}:\U^{\perp} \mono \D$). The composition $\iota_\U \circ \tau_{\U}^{\leq}:\D \rightarrow \D$ (resp. $\iota_{\U^{\perp}} \circ \tau_{\U}^{>}:\D \rightarrow \D$), which we will still denote by $\tau_{\U}^{\leq}$ (resp. $\tau_{\U}^{>}$) and it is called the \emph{left truncation} (resp. \emph{right truncation}) functor associated to the t-struture $(\U,\U^{\perp}[1]).$ The full subcategory $\Ha=\U \cap \U^{\perp}[1]$ of $\D$ is called the \emph{heart} of the t-structure and it is an abelian category, where the short exact sequences are the triangles in $\D$ having their three terms in $\Ha.$ In particular, we have $\Ext^{1}_{\Ha}(M,N)=\D(M,N[1]),$ for all objects $M$ and $N$ in $\Ha$. Moreover, the canonical morphism $\Ext^{2}_{\Ha}(M,N) \rightarrow \D(M,N[2])$ is a monomorphism in $Ab$, for all objects $M, N \in \Ha$ (see \cite[Remarque 3.1.17]{BBD}). 

\vspace{0,3 cm}

The kernel and cokernel of a morphism $f:M \rightarrow N$ on the heart $\Ha$ are computed as follows: we complete $f$ to a triangle in $\D$ and consider the following diagram, where the row and column are triangles:

$$\xymatrix{ &&\tau_{\U}^{\leq}(Z[-1])[1]  \ar[d] \\   M  \ar[r]^{f} &  N \ar[r] & Z \ar[r]^{+} \ar[d] & \\ && \tau_{\U}^{>}(Z[-1])[1] \ar[d]^{+}  \\ &&&&} $$

From the octahedral axiom, we obtain the following triangles in $\D$ 

$$\xymatrix{\tau_{\U}^{\leq}(Z[-1]) \ar[r] &M  \ar[r]^{p_f} & I \ar[r] ^{+} & & I \ar[r]^{\iota_f} & N \ar[r]&  \tau_{\U}^{>}(Z[-1])[1] \ar[r]^{\hspace{0,8 cm}+} & },$$ 
 where all terms are in $\Ha$ (with $\iota_f \circ p_f=f$). Then we have that $\Ker_{\Ha}(f)=\tau_{\U}^{\leq}(Z[-1])$ and $\Coker_{\Ha}(f)=\tau_{\U}^{>}(Z[-1])[1].$

Recall that if $\D$ and $\A$ are a triangulated and an abelian category,  respectively, then an additive functor $H:\D\longrightarrow\A$ is called \emph{cohomological} when  any triangle $X\stackrel{u}{\longrightarrow}Y\stackrel{v}{\longrightarrow}Z\stackrel{+}{\longrightarrow}$ in $\D$ induces an exact sequence $H(X)\stackrel{H(u)}{\longrightarrow}H(Y)\stackrel{H(v)}{\longrightarrow}H(Z)$ in $\A$. The following proposition,  shown by Beilinson, Bernstein and Deligne in \cite{BBD}, associates to each t-structure in a triangulated category  an intrinsic homology theory. 

\begin{proposition}
Let $(\D,?[1])$ be a triangulated category. If $\sigma=(\U,\U^{\perp}[1])$ is a t-structure in $\D$, then the assignments $X \leadsto \tau_{\U}^{\leq}(\tau_{\U}^{>}(X[-1])[1])$ and $X \leadsto \tau_{\U}^{>}(\tau_{\U}^{\leq}(X)[-1])[1]$ define two naturally isomorphic functors from $\D$ to $\Ha$, which are cohomological.
In the sequel we will fix a (cohomological) functor $H^{0}_{\sigma}: \D \rightarrow \Ha$ that is naturally isomorphic to those two functors. 
\end{proposition}

\begin{examples} \label{examples-t-structures} \normalfont
The following examples of t-structures will be of great interest in this paper. 
\begin{enumerate}
\item Let $\A$ be an abelian category for which  $\D(\A)$  exists i.e.  $\D(\A)$ has $Hom$ sets. For each $m\in \mathbb{Z}$, we will denote by $\D^{\leq m}(\A)$ (resp. $\D^{\geq m}(\A)$) the full subcategory of $\D(\A)$ consisting of the cochain complexes $X$ such that $H^{k}(X)=0$, for all $k>m$ (resp. $k < m$). Moreover, we put $\D^{[a,b]}(\A):=\D^{\leq b}(\A) \cap \D^{\geq a}(\A)$ for any integers $a$ and $b$. Then, the pair $(\D^{\leq m}(\A),\D^{\geq m}(\A))$ is a t-structure in $\D(\A)$ whose heart is equivalent to $\A.$ The corresponging left and right truncation functors will be denoted by $\tau^{\leq m}: \D(\A) \rightarrow \D(\A)$ and $\tau^{>m}:\D(\A) \rightarrow \D(\A)$, respectively. For the case $m=0$, the corresponding t-structure is known as the \emph{canonical t-structure} in $\D(\A).$

\item (Happel-Reiten-Smal\o ) Let $\A$ be an abelian category for which $\D(\A)$ exists  and let $\mathbf{t}=(\T,\F)$ be a torsion pair in $\A$. The classes $\U_{\mathbf{t}}:=\{X \in \D^{\leq 0} (\A): H^{0}(X)\in \T\}$ and $\W_{\mathbf{t}}=\{X \in \D^{\geq -1}(\A): H^{-1}(X) \in \F\}$ give rise a t-structure in $\D(\A)$, concretely, $(\U_{\mathbf{t}},\W_{\mathbf{t}})=(\U_{\mathbf{t}},\U_{\mathbf{t}}^{\perp}[1])$. It is called the \emph{Happel-Reiten-Smal\o \ t-structure} associated to $\mathbf{t}$.  We will denote its heart by $\Ht$. In next subsection we relax the hypothesis that $\D(\A)$ has $Hom$ sets, showing that the formation of $\Ht$ from $\A$ and $\mathbf{t}$ is still possible sometimes.

\item Let $\D$ be a triangulated category with coproducts.  An object $X$ in $\D$ is called \emph{compact}, when the functor $\D(X,?):\D \rightarrow Ab$ preserves  coproducts. If $\mathcal{S}$ is a set of compact objects in $\D$, then the pair $(_{}^\perp(\mathcal{S}^{\perp_{\leq 0}}),\mathcal{S}^{\perp_{<0}})$ is a t-structure in $\mathcal{D}$ (see \cite[Theorem 12.1]{KN}).  A t-structure is said to be \emph{compactly generated} if it is obtained in this way from a set of compact objects. We put $\text{aisle} <\mathcal{S}>:=_{}^\perp(\mathcal{S}^{\perp_{\leq 0}})$, which is the smallest full subcategory of $\mathcal{D}$ that contains $\mathcal{S}$ and is closed under coproducts, extensions and non-negative shifts.
\end{enumerate}
\end{examples}

\subsection{The Happel-Reiten-Smal\o \ (HRS) tilting process}
  
  This process stems from the seminal work in \cite{BBD} and was fully developed in \cite{HRS}. In our treatment here we will work in a general framework, by allowing ourselves the freedom of working for the moment with big  triangulated categories. For that, let us define a \emph{big abelian group} as a (proper) class $\mathbb{A}$ of elements together with a map between classes $+:\mathbb{A}\times\mathbb{A}\longrightarrow\mathbb{A}$, $(a,b)\rightsquigarrow a+b$,  called sum, that satisfies the axioms of usual abelian groups. When we add another map of classes $\cdot :\mathbb{A}\times\mathbb{A}\longrightarrow\mathbb{A}$, $(a,b)\rightsquigarrow ab$, called  multiplication,  such that $+$ and $\cdot$ satisfy the usual axioms of a(n associative unital) ring, we call $\mathbb{A}$ a \emph{big ring}. Homomorphisms of big abelian groups and of big rings are defined as for usual groups and rings.  
  
    A \emph{big preadditive category} $\mathcal{A}$ consists then of a class of objects $\text{Ob}(\mathcal{A})$ and, for each pair $(X,Y)\in\text{Ob}(\mathcal{A})\times\text{Ob}(\mathcal{A})$, a big abelian group of morphisms $\mathcal{A}(X,Y)$, both satisfying the usual axioms for a preadditive category. When such a big preadditive category satisfies the usual axioms of additive, abelian or triangulated categories, we will call it a \emph{big additive}, \emph{big abelian} or \emph{big triangulated} category. 
    
     Note that the use of these big categories is also implicit in \cite{HRS} since the authors work with the bounded derived category $\mathcal{D}^b(\A)$ of an abelian category, which need not have $Hom$ sets as Example \ref{ex.Casacuberta-Neeman} below shows.  This use of big additive categories allows more flexibility in  the HRS tilting process and, increasing the universe if necessary, will pose no set-theoretical problems.  In the rest of the paper we adopt the convention that the term \emph{category} means a category with $Hom$ sets.  So the expression 'is a category' will mean 'is a category with $Hom$ sets'. 
    
 The following is an expansion of an example in \cite{CN}.
    
\begin{example} \label{ex.Casacuberta-Neeman} \normalfont
For any big ring $\mathbb{A}$ appearing in this example, we consider the category $\mathbb{A}-\text{Mod}$ of small $\mathbb{A}$-modules.  Its objects are pairs $(M,f)$ consisting of an abelian group $M$
 together with a ring homomorphism  $f:\mathbb{A}\longrightarrow\text{End}_\mathbb{Z}(M)$. The morphisms $(M,f)\longrightarrow (N,g)$ are the homomorphisms $\varphi :M\longrightarrow N$ of abelian groups such that $\varphi\circ f(a)=g(a)\circ\varphi$, for all $a\in\mathbb{A}$. 

Given the class $I$  of all ordinals and an isomorphic class of variables $\{X_\alpha\text{: }\alpha\in I\}$, we shall associate two big rings. The first one is the big ring of polynomials  $R:=\mathbb{Z}[X_\alpha\text{: }\alpha\in I]$, as defined in $\cite{CN}$, where we have slightly changed the notation of that paper. The second one is
the big free ring  $A:=\mathbb{Z}<X_\alpha\text{: }\alpha\in I>$, i.e. its elements are finite $\mathbb{Z}$-linear combinations of words on the alphabet $\{X_\alpha:  \ \alpha\in I\}\cup\{1\}$, and the multiplication extends by $\mathbb{Z}$-linearity the obvious juxtaposition of words with $1$ as multiplicative identity. Note that giving a small $A$-module $(M,f)$ amounts to giving an $I$-indexed class $(f_\alpha )_{\alpha\in I}$ in $\text{End}_\mathbb{Z}(M)$, namely $f_\alpha =f(X_\alpha )$ for all $\alpha\in I$, and theses $f_\alpha$ are required to commutate in the case that $(M,f)$ is a small $R$-module. The morphisms $\varphi :(M,f)\longrightarrow (N,g)$ (in $R-\text{Mod}$ or $A-\text{Mod}$) are just the homomorphisms of abelian groups $\varphi :M\longrightarrow N$ such that $\varphi\circ f_\alpha =g_\alpha\circ\varphi$, for all $\alpha\in I$. 
 
We have an obvious fully faithful exact embedding $u:R-\text{Mod}\hookrightarrow A-\text{Mod}$. A \emph{trivial small ($R$-  or $A$-)module} is an $(M,f)$ such that $f_\alpha =0$, for all $\alpha\in I$. The forgetful functors $R-\text{Mod}\longrightarrow\text{Ab}$ and $A-\text{Mod}\longrightarrow\text{Ab}$ clearly induce an equivalence between each of the respective  subcategories of trivial small modules and $\text{Ab}$.
  By \cite[Lemma 1.1]{CN}, we know that $\mathcal{D}(R-\text{Mod})(\mathbb{Z},\mathbb{Z}[1])\cong\text{Ext}_{R}^1(\mathbb{Z},\mathbb{Z})$ is not a set, where $\mathbb{Z}$ is the trivial small $R$-module associated to $\mathbb{Z}$. Since we have an obvious inclusion $$\mathcal{D}(R-\text{Mod})(\mathbb{Z},\mathbb{Z}[1])\cong\text{Ext}_{R}^1(\mathbb{Z},\mathbb{Z})\hookrightarrow\text{Ext}_{A}^1(\mathbb{Z},\mathbb{Z})\cong\mathcal{D}(A-\text{Mod})(\mathbb{Z},\mathbb{Z}[1])$$ we deduce that both $\mathcal{D}(R-\text{Mod})$ and  $\mathcal{D}(A-\text{Mod})$, and even  $\mathcal{D}^b(R-\text{Mod})$ and  $\mathcal{D}^b(A-\text{Mod})$, are big triangulated categories, i.e. do not have $\text{Hom}$ sets. 
\end{example}

  The following result is the version for big triangulated categories of \cite[Proposition 3.1.1 and 3.1.4]{Ma2}. Recall that a t-structure $\tau 
=(\U,\W)$ is \emph{left (resp. right) nondegenerate} when 
$\bigcap_{n\in\mathbb{Z}}\U[n]=0$ (resp. 
$\bigcap_{n\in\mathbb{Z}}\W[n]=0$), and it is called 
\emph{nondegenerate} when it is left and right nondegenerate. Mattiello's proof is valid here and proves the result, except for the nondegeneracy of $\tau_\mathbf{t}$ and the final part of the statement, that are explicitly proved.

\begin{proposition} \label{prop.HRS tilt}
Let $\mathcal{D}$ any big triangulated category, let $\tau=(\mathcal{U},\mathcal{W})$ be a nondegenerate t-structure in it and denote by $\A$ its heart, which is then a big abelian category, and denote by $H_\tau^0:\mathcal{D}\longrightarrow\mathcal{A}$ the associated cohomological functor.  If $\mathbf{t}:=(\mathcal{T},\mathcal{F})$ is a torsion pair in $\mathcal{A}$, then the pair $\tau_\mathbf{t}:=(\mathcal{U}_\mathbf{t},\mathcal{W}_\mathbf{t})$ given by the following classes is again a nondegenerate t-structure in $\D$:

$$\mathcal{U}_\mathbf{t}=\{X\in\mathcal{D}\text{: }H_\tau^k(X)=0\text{, for }k>0\text{,  and }H_\tau^0(X)\in\mathcal{T} \} $$

$$\mathcal{W}_\mathbf{t}=\{Y\in\mathcal{D}\text{: }H_\tau^k(Y)=0\text{, for } k<-1\text{,  and }H_\tau^{-1}(Y)\in\mathcal{F} \}. $$

Moreover, the pair $\overline{\mathbf{t}}:=(\mathcal{F}[1],\mathcal{T})$ is a torsion pair in the heart $\Ht$ of  $\tau_\mathbf{t}$ and, for each $M\in\Ht$, the associated torsion sequence is  of the form $0\rightarrow H_\tau^{-1}(M)[1]\longrightarrow M\longrightarrow H_\tau^0 (M)\rightarrow 0$. 
\end{proposition}
\begin{proof} For the nondegeneracy of $\tau_\mathbf{t}$, note that we clearly have that $\bigcap_{n\in\mathbb{Z}}\mathcal{U}_{\mathbf{t}}$ consists of the objects $X$ such that $H_\tau^k(X)=0$, for all $k\in\mathbb{Z}$. The nondegeneracy of $\tau$ then implies that $X=0$ (see, e.g., \cite[Lemma 3.3]{NSZ}, adapted to big triangulated categories). This gives the left nondegeneracy of $\tau_\mathbf{t}$ and the right nondegeneracy follows dually. 

Recall that, due to the nondegeneracy of $\tau$, we have that $\U =\{U\in\mathcal{D}\text{: }H_\tau^i(U)=0\text{, for all }i>0\}$ and  $\W =\{W\in\mathcal{D}\text{: }H_\tau^i(W)=0\text{, for all }i<0\}$ (see \cite[Proposition 1.3.7]{BBD}), and hence $\A$ consists of the objects $A\in\D$ such that $H_\tau^i(A)=0$, for  all integers $i\neq 0$. If now $M\in\Ht =\U_\mathbf{t}\cap\W_\mathbf{t}$ is any object in the heart of $\tau_\mathbf{t}$, then $H_\tau^i(M[-1])=H_\tau^{i-1}(M)$ is zero, for all $i<0$, so that  $M[-1]\in\W$ and hence  $H_\tau^{-1}(M)=H_\tau^0(M[-1])\cong\tau_\mathcal{U}^{\leq}(M[-1])$. The associated  truncation triangle with respect to $\tau$ is then  $$H_\tau^{-1}(M)\longrightarrow M[-1]\longrightarrow T[-1]\stackrel{+}{\longrightarrow}, $$ where $T\in\W$. By shift, we get the triangle  $$H_\tau^{-1}(M)[1]\longrightarrow M\longrightarrow T\stackrel{+}{\longrightarrow}. $$
By taking the long exact sequence in $\A$ obtained by applying to the last triangle the functor $H_\tau^0:\D\longrightarrow\A$, we readily see that $H_\tau^i(T)=0$, for all $i\neq 0$, and hence $T\in\A$, and that the induced morphism $H_\tau^0(M)\longrightarrow H_\tau^0(T)\cong T$ is an isomorphism. We then have a triangle $$H_\tau^{-1}(M)[1]\longrightarrow M\longrightarrow H_\tau^0(M)\stackrel{+}{\longrightarrow}$$ in $\D$ with its three terms in $\Ht$, and hence it gives a short exact sequence in this latter category. 

\end{proof}

\begin{definition} \label{def.HRS-tilted t-structure} \normalfont
The t-structure $\tau_\mathbf{t}$ of last proposition is said to be the \emph{HRS-tilt of $\tau$ with respect to $\mathbf{t}$}.  The \emph{HRS process} in $\mathcal{D}$ is a `map' $\Phi_\mathcal{D}$ defined on the class of pairs $(\tau ,\mathbf{t})$, where $\tau$ is a nondegenerate t-structure in $\mathcal{D}$ and $\mathbf{t}$ is a torsion pair in the heart $\mathcal{H}_\tau$ of $\tau$. It is defined by $\Phi_\mathcal{D} (\tau ,\mathbf{t})=(\tau_\mathbf{t},\bar{\mathbf{t}})$. 
\end{definition}

The following is now a very natural question. 

\begin{question}
In the situation of Proposition \ref{prop.HRS tilt}, assume that $\A$ has $Hom$ sets. When is it true that also $\mathcal{H}_\mathbf{t}$ has $Hom$ sets?
\end{question}

The following is the answer:

\begin{proposition} \label{prop.HRS-tilt-big-triangcats}
Let $\tau=(\mathcal{U},\mathcal{W})$ be a nondegenerate t-structure in the big triangulated category $\mathcal{D}$ such that its heart $\A$ is a category. Let $\mathbf{t}=(\mathcal{T},\mathcal{F})$ be a torsion pair in $\A$. The following assertions are equivalent:

\begin{enumerate}
\item The heart $\mathcal{H}_\mathbf{t}$ of the tilted t-structure $\tau_\mathbf{t}$ is a category.
\item $\D(T,F[1])\cong\Ext_{\A}^1(T,F)$ is a set, for all $T\in\mathcal{T}$ and all $F\in\mathcal{F}$. 
\end{enumerate}
In such case, if one puts $\Phi_\mathcal{D}^n(\tau,\mathbf{t})=:(\tau_n,\mathbf{t}_n)$,  where $\Phi_\mathcal{D}^n=\Phi_\mathcal{D}\circ\stackrel{n}{\cdots}\circ\Phi_\mathcal{D}$, then the heart of $\tau_n$ has $Hom$ sets, for each $n>0$. 
\end{proposition}
\begin{proof}
$(1)\Longrightarrow (2)$ is clear since, by \cite{BBD}, we have an isomorphism $\Ht(T,F[1])\cong\D(T,F[1])$.

$(2)\Longrightarrow (1)$   For simplicity,   put $H^k:=H_\tau\circ (?[k])$ for each $k\in\mathbb{Z}$.  Let $M,N$ be objects of $\mathcal{H}_\mathbf{t}$. We have a triangle $H^{-1}(N)[1]\longrightarrow N\longrightarrow H^0(N)\stackrel{+}{\longrightarrow}$ (*) in $\D$ (see Proposition  \ref{prop.HRS tilt}).  An application of the cohomological functor $\D(M,?)$ from $\D$ to the category $\text{AB}$ of big abelian groups gives an exact sequence $$\D(M,H^{-1}(N)[1])\longrightarrow \D(M,N)\longrightarrow \D(M,H^0(N)),   $$ Therefore the proof is reduced to check that $\D(M,N)$ is a set when $N=F[1]$, for some $F\in\mathcal{F}$, or $N=T\in\mathcal{T}$, for some $T\in\mathcal{T}$. 

Suppose that $N=T\in\mathcal{T}$. Taking the triangle (*) with $M$ instead of $N$ and applying to it the cohomological functor $\D(?,T)$, we obtain an exact sequence $$\D(H^0(M),T)\longrightarrow\D(M,T)\longrightarrow\D(H^{-1}(M)[1],T)=0. $$ But $\D(H^0(M),T)\cong\A(H^0(M),T)$ is a set by the hypothesis on $\A$. It then follows that $\D(M,T)$ is a set. 

Suppose that $N=F[1]\in\mathcal{F}[1]$ and apply $\D(?,F[1])$ to the triangle of the previous paragraph. We get the exact sequence 
$$0=\D(H^{-1}(M)[2],F[1])\rightarrow \D(H^0(M),F[1])\rightarrow \D(M,F[1])\rightarrow \D(H^{-1}(M)[1],F[1]).$$ But we have that $\D(H^{-1}(M)[1],F[1])\cong \D(H^{-1}(M),F)\cong \A(H^{-1}(M),F)$, which is a set due to the hypothesis on $\A$. Then $\D(M,F[1])$ is a set since, by hypothesis, $\D(H^0(M),F[1])$ is a set. 

Next we consider the tilted torsion pair $\bar{t}=(\mathcal{F}[1],\mathcal{T})$ in $\Ht$. Then, by [BBD], we have that $$\Ext_{\Ht}^1(F[1],T)\cong \D(F[1],T[1])\cong \D(F,T)\cong \A(F,T),$$ which is a set due to the hypothesis on $\A$. Replacing now $\A$ by $\Ht$ in the equivalence of assertions 1 and 2, we get that the tilted t-structucture $\tau_{\overline{\mathbf{t}}}$ with respect to $\overline{\mathbf{t}}$ has a heart $\mathcal{H}_{\overline{\mathbf{t}}}$ which is a category. The last statement of the proposition is then clear. 
\end{proof}

As a particular case, when the big triangulated category is $\D=\D(\A)$, where $\A$ is an abelian category, one can consider the canonical t-structure $\tau=(\mathcal{D}^{\leq 0}(\A),\D^{\geq 0}(\A))$ as initial one. Its heart is $\A$ and the functor $H_\tau^k:\D(\A)\longrightarrow\A$ is the classical $k$-th cohomology functor. As a direct consequence of last proposition, 
we get: 

\begin{corollary} \label{cor.heart-genuine-category}
Let $\A$ be an abelian category and let $\mathbf{t}=(\mathcal{T},\mathcal{F})$ be a torsion pair in $\A$. The following assertions are equivalent:
\begin{enumerate}
\item The heart $\mathcal{H}_\mathbf{t}$ is a  category, i.e. with $Hom$ sets.
\item $\Ext_{\A}^1(T,F)$ is a set (as opposite to a proper class), for all $T\in\mathcal{T}$ and $F\in\mathcal{F}$.
\end{enumerate}
\end{corollary}

The following shows that condition 2 of last corollary need not hold in general.

\begin{example} \label{ex.Casacuberta-Neeman2} \normalfont
Let $A=\mathbb{Z}<X_\alpha\text{: }\alpha\in I>$ be as in Example \ref{ex.Casacuberta-Neeman} and consider the torsion pair $\mathbf{t}$ in $A-\text{Mod}$ generated by the trivial small $A$-module $\mathbb{Z}$, i.e.  $\mathbf{t}=(\T,\F):=(_{}^\perp(\mathbb{Z}^\perp),\mathbb{Z}^\perp)$. The heart $\mathcal{H}_\mathbf{t}$ of the associated Happel-Reiten-Smal\o \          t-structure in $\mathcal{D}(A-\text{Mod})$ is a big abelian category, i.e. it does not have $\text{Hom}$ sets. 
\end{example}
\begin{proof}
In order to help with the intution, given a small $A$-module $(M,f)$,  we put $X_\alpha m:=f(X_\alpha )(m)$, for all $m\in M$ and $\alpha\in I$. A trivial small $A$-module is then one such that $X_\alpha M=0$, for all $\alpha\in I$. It is clear that $(F,f)$ is in $\mathcal{F}$ if and only if it does not contain any nonzero trivial $A$-submodule. 

We next consider two (big) ideals $J$ and $J_\beta$ of $A$, the second one depending on the ordinal $\beta\in I\setminus\{0\}$. The ideal $J$ is generated by all variables $X_\alpha$, with $\alpha\neq 0$, while $J_\beta$ is generated by the set $$\{X_\alpha\text{: }\alpha\in I\setminus\{0,\beta \}\}\cup\{X_\beta^2,X_\beta X_0\}. $$ It is clear that $A/J$ is isomorphic to $\mathbb{Z}[X_0]$ and that it is an $A$-module in $\mathcal{F}$ for none of its nonzero elements is annihilated by $X_0$. On the other hand,   the underlying abelian group of $A/J_\beta$ is free with basis $\{X_0^n\text{: }n\in\mathbb{N}\}\cup\{X_0^nX_\beta\text{: }n\in\mathbb{N}\}$. We then get a morphism $\varphi :A/J\oplus A/J\cong \mathbb{Z}[X_0]\oplus\mathbb{Z}[X_0]\longrightarrow A/J_\beta$  in $A-\text{Mod}$,   that takes $(P(X_0),Q(X_0))\rightsquigarrow P(X_0)X_0+Q(X_0)X_\beta$ and whose cokernel is the trivial small $A$-module $\mathbb{Z}$.  Note that if $(P(X_0),Q(X_0))\in\text{Ker}(\varphi )$, so that $P(X_0)X_0+Q(X_0)X_\beta =0$, we then get that $(P(X_0),Q(X_0))=(0,0)$ by considering the $\mathbb{Z}$-basis of $A/J_\beta$ given above.  Therefore $\varphi$ is a monomorphism. 

The last paragraph shows that, for each $\beta\in I\setminus\{0\}$, we have an exact sequence $$0\rightarrow A/J\oplus A/J\longrightarrow A/J_\beta\longrightarrow\mathbb{Z}\rightarrow 0 $$ in $A-\text{Mod}$. It is clear that, for $\beta\neq\gamma$ in $I\setminus\{0\}$, the left $A$-modules $A/J_\beta$ and $A/J_\gamma$ cannot be isomorphic since they have different annihilators. It then follows that $\text{Ext}_A^1(\mathbb{Z},A/J\oplus A/J)$ is not a set, which implies that $\mathcal{H}_\mathbf{t}$ does not have $\Hom$ sets by Corollary \ref{cor.heart-genuine-category}. 
\end{proof}

 \begin{lemma}\label{new lemma} If $V$ is an object of an abelian category $\A$ such that all coproducts of copies of $V$ exist in $\A$ and $\Ext^{1}_{\A}(V,X)$ is a set, for all $X\in \A$, then the canonical map $\Ext^{1}_{\A}(V^{(I)},X) \longrightarrow \Ext^{1}_{\A}(V,X)^{I}$ is a monomorphism, and hence $\Ext^{1}_{\A}(V^{(I)},X)$ is a set,   for each $X\in \A$ and $I$ set. 
\end{lemma}
\begin{proof} Let $X$ be an object in $\A$ and let $I$ be a set. We consider $\epsilon: 0 \longrightarrow X \longrightarrow M \stackrel{f}{\longrightarrow} V^{(I)} \longrightarrow 0$ an extension in $\A$ such that $\phi(\epsilon)=0,$ where $\phi$ denote the respective canonical assignment $\Ext^{1}_{\A}(V^{(I)},X) \longrightarrow \Ext^{1}_{\A}(V,X)^{I}.$ We will show that $\epsilon$ is a split exact sequence. Indeed, for each $j\in I$, the $j$-th inclusion $\iota_{j}:V \longrightarrow V^{(I)}$ factors through $f:M \longrightarrow V^{(I)}$ since $\phi(\epsilon)=0$. We then get a morphism $g_j:V \longrightarrow M$ such that $f \circ g_j=\iota_j.$ When $j$ varies in $I$, the universal property of coproducts yields a unique morphism $g:V^{(I)} \longrightarrow M$ such that $g \circ \iota_j=g_j,$ for all $j \in I.$ It follows that $f \circ g \circ \iota_j=\iota_j$, for all $j \in I,$ which implies that $f \circ g=1_{V^{(I)}}$. Then $\epsilon$ is a split sequence, and so $\epsilon =0$ in $\Ext^{1}_{\A}(V^{(I)},X) $.
\end{proof}

 In order to exhibit a very useful direct consequence of Corollary \ref{cor.heart-genuine-category}, we need to introduce a few subcategories associated to an object that will play an important role through the paper.

\begin{definition}\label{def. Pres V} \normalfont
Let $\A$ be an abelian category and let $X$ and $V$ be objects of $\A$, where we asume that all (set-indexed) coproducts of copies of $V$ exist in $\A$. We will say that $X$ is \emph{$V$-generated} (resp. \emph{$V$-presented}) when there is an epimorphism of the form $V^{(I)} \epi X$ (resp. an exact sequence $V^{(J)} \longrightarrow V^{(I)} \epi X$) for some set $I$ (resp. sets $I$ and $J$).  We will denote by $\Gen(V)$ and $\Pres(V)$ the classes of $V$-generated and $V$-presented objects, respectively. 

When  $Q\in\text{Ob}(\A)$ is such that all products of copies of $Q$ exist in $\A$,  we get the dual notions of  $Q$-\emph{cogenerated} and $Q$-\emph{copresented} object, and the corresponding subcategories  $\Cogen(Q)$ and $\Copres(Q)$.
\end{definition}

\begin{corollary} \label{cor.heart-quasi-(co)tilting pair}
In the situation of last corollary, suppose that  $\mathcal{T}=\Pres (V)$ (resp. $\mathcal{F}=\Copres(Q)$), for some object $V$ (resp. $Q$) of $\A$ such that all coproducts (resp. products) of copies of $V$ (resp. $Q$) exist in $\A$. The heart $\mathcal{H}_\mathbf{t}$ is a category if, and only if, $\text{Ext}_{\A}^1(V,F)$ (resp. $\Ext_{\A}^1(T,Q)$) is a set, for all $F\in\mathcal{F}$ (resp. $T\in\mathcal{T}$). 
\end{corollary}
\begin{proof}
The statement for $Q$ is dual of the statement for $V,$ so we just prove the latter one. Note that if $\T=\text{Pres}(V)$ then the equality $\T=\text{Gen}(V)$ also holds.  Let's take any $T\in\mathcal{T}$ and consider an exact sequence $0\rightarrow T'\longrightarrow V^{(I)}\longrightarrow T\rightarrow 0$, with $T'\in\mathcal{T},$ for some set $I$. If $F\in\mathcal{F}$ is any object and we apply the contravariant functor $\A(?,F)$, we obtain the exact sequence 
$$0=\A(T',F)\longrightarrow\Ext_{\A}^1(T,F)\longrightarrow\text{Ext}_{\A}^1(V^{(I)},F)$$ and, by Lemma \ref{new lemma}, we deduce that  $\Ext_{\A}^1(T,F)$ is a set, for all $T\in\mathcal{T}$. Now  Corollary \ref{cor.heart-genuine-category} applies. 

\end{proof}

This allows us a re-interpretation of the HRS process, where the map $\Phi$ acts instead on pairs $(\A,\mathbf{t})$, where $\A$ is an abelian category and $\mathbf{t}=(\T,\F)$ is torsion pair in $\A$ satisfying the equivalent conditions of last corollary. Concretely:

\begin{definition} \label{def.HRS tilting process} \normalfont
A torsion pair $\mathbf{t}=(\T,\F)$ in an abelian category $\A$ will be an \emph{adequate torsion pair} when $\text{Ext}_\A^1(T,F)$ is a set, as opposite to a proper class, for all $T\in\T$ and $F\in\F$. We will denote by $(\mathbf{AB},\mathbf{tor})$ the class of pairs $(\A,\mathbf{t})$ consisting of an abelian category $\A$ and an adequate torsion pair $\mathbf{t}$ in it. The \emph{Happel-Reiten-Smal\o \ (HSR) tilting process} is the map $\Phi :(\mathbf{AB},\mathbf{tor})\longrightarrow(\mathbf{AB},\mathbf{tor})$ given by  $\Phi[(\A,\mathbf{t})]=(\Ht,\bar{\mathbf{t}})$, where $\Ht$ is the heart of the HRS tilt  $\tau_\mathbf{t}$ of the canonical t-structure of $\mathcal{D}(\A)$ with respect to $\mathbf{t}$ (see Definition \ref{def.HRS-tilted t-structure}).
\end{definition}

In particular if, under the hypotheses of last corollary, we assume that $\Ext_\A^1(T,F)$ is a set, for all $T\in\mathcal{T}$ and all $F\in\mathcal{F}$, then, by HRS-tilting iteration, one gets
gets the following diagram of abelian categories, all of them with $Hom$ sets,  and torsion pairs: 

$$\xymatrix{ \A &  \Ht & \Ha_{\overline{\mathbf{t}}} \\ \mathbf{t}=(\T,\F) \ar@{~>}[r] & \overline{\mathbf{t}}=(\F[1],\T[0]) \ar@{~>}[r] & \overline{\overline{\mathbf{t}}}=(\T[1],\F[1])}$$

The following is \cite[Proposition 3.2]{HRS}.

\begin{proposition} \label{prop.HRS1}
Let $\Phi :(\mathbf{AB},\mathbf{tor})\longrightarrow (\mathbf{AB},\mathbf{tor})$ be the HRS tilting process map (see Definition \ref{def.HRS tilting process}). Let  $(\A,\mathbf{t})$ be  in $(\mathbf{AB},\mathbf{tor})$ and put $(\mathcal{B},\bar{\mathbf{t}}):=\Phi [(\A,\mathbf{t})]$. The torsion pair $\mathbf{t}$ is faithful (resp. co-faitful) if, and only if, $\bar{\mathbf{t}}$ is co-faithful (res. faithful). In particular we have induced maps by restriction $\xymatrix{(\mathbf{AB},\mathbf{tor}_{faithful})\ar@<1ex>[r]^{\phi} & \ar@<1ex>[l]^{\phi} (\mathbf{AB},\mathbf{tor}_{cofaithful})}$, where $\mathbf{tor}_{faithful}$ (resp. $\mathbf{tor}_{cofaithful}$) denotes the class of adequate faithful (resp. co-faithful) torsion pairs. 
  
\end{proposition}

Furthermore, we have a triangle functor $G:\D^{b}(\Ht) \rightarrow \D^{b}(\A)$ whose restriction to $\Ht$ is naturally isomorphic to the inclusion functor $\Ht\hookrightarrow\D^{b}(\A)$ (see \cite[Proposition 3.1.10]{BBD}). The  functor $G$ is usually called the \emph{realization functor}.  The following proposition shows that in some cases the heart $\Ha_{\overline{\mathbf{t}}}$ is equivalent to $\A[1]$. 

\begin{proposition}{[HRS, Proposition 3.4]\label{prop. Equiva A[1]Ht}} 
Let $\A$ be an abelian category and let $\mathbf{t}=(\T,\F)$ be a torsion pair in $\A$ such that $\T$ is a cogenerating class. The following assertions  hold:
\begin{enumerate}
\item If $\Ht$ has enough projectives, then $\Ha_{\overline{\mathbf{t}}}\cong \A[1]$ via the realization functor;
\item If $\A$ has enough injectives, then $\Ha_{\overline{\mathbf{t}}}\cong \A[1]$ via the realization functor.
\end{enumerate}
In particular, whenever $\A$ is a category with enough projectives or with enough injectives, $\Phi^2 [(\A,\mathbf{t})]\cong (\A,\mathbf{t})$. That is, if $\Phi^2 [(\A,\mathbf{t})]=(\A',\mathbf{t}')$, then there is an equivalence of categories $F:\A\stackrel{\cong}{\longrightarrow}\A'$ which takes $\mathbf{t}$ to $\mathbf{t}'$. 
\end{proposition}

 \section{Projective and injective objects in the heart. Quasi-(co)tilting torsion pairs}
 
 \subsection{Quasi-(co)tilting objects and torsion pairs}

We start with two auxiliary lemmas.

\begin{lemma} \label{lem.coproducts-Pres(X)}
Let $\A$ be an abelian category and let $X$  be an object such that all coproducts (resp. products) of copies of $X$  exist in $\A$. Then  all coproducts (resp. products) of objects in  $\Pres(X)$ (resp. $\Copres (X)$) exist in $\A$. 
\end{lemma}
\begin{proof}
The result for $\Copres (X)$ is dual of the one for $\Pres(X)$. We just do the latter one. 
Let $(T_\lambda )_{\lambda\in\Lambda}$ be a family in $\Pres(X)$, fix an exact sequence $X^{(J_\lambda)}\stackrel{f_\lambda}{\longrightarrow}X^{(I_\lambda)}\longrightarrow T_\lambda\rightarrow 0$, with sets $J_\lambda$ and $I_\lambda$, for each $\lambda\in\Lambda$. We then get an induced exact sequence $\coprod_{\lambda\in\Lambda}X^{(J_\lambda )}\stackrel{\coprod f_\lambda}{\longrightarrow}\coprod_{\lambda\in\Lambda}X^{(I_\lambda )}\longrightarrow\text{Coker} (\coprod f_\lambda)\rightarrow 0$. This gives the following commutative diagram of functors $\A\longrightarrow\text{Ab}$, with exact rows:

$$\xymatrix{0 \ar[r] & \A(\Coker(\coprod f_{\lambda}), ?) \ar[r] \ar[d]^{\alpha} &  \A(\coprod_{\lambda \in \Lambda} X^{(I_{\lambda})}, ?) \ar[r] \ar[d]^{\wr} &  \A( \coprod_{\lambda \in \Lambda} X^{(J_{\lambda})},?) \ar[d]^{\wr}\\0 \ar[r] & \prod_{\lambda \in \Lambda} \A(T_{\lambda}, ?) \ar[r] & \prod_{\lambda \in \Lambda} \A(X^{(I_{\lambda})}, ?) \ar[r] & \prod_{\lambda \in \Lambda} \A(X^{(J_{\lambda})},?) }$$

where the left vertical arrow $\alpha$ exists by the universal property of kernels in $\text{Ab}$. By definition of coproducts, the two right vertical arrows are isomorphisms, which in turn implies that $\alpha$ an isomorphism. By Yoneda's lemma, for each $\mu\in\Lambda$, the composition $\mathcal{A}(\text{Coker}(\coprod f_\lambda),?)\stackrel{\alpha}{\longrightarrow}\prod_{\lambda\in\Lambda}\A(T_\lambda ,?)\stackrel{\pi_\mu}{\longrightarrow}\A(T_\mu ,?)$, where $\pi_\mu$ is the projection,  is of the form $u_\mu^*=\A(u_\mu,?)$, for an unique morphism $u_\mu :T_\mu\longrightarrow\text{Coker}(\coprod f_\lambda)$.  It immediately follows that $\text{Coker}(\coprod f_\lambda)$ together with the morphisms $(u_\lambda :T_\lambda\longrightarrow\text{Coker}(\coprod f_\lambda))$ is the coproduct of the $T_\lambda$ in $\A$. 

\end{proof}

\begin{lemma} \label{Gen(V) torsion class}
 Let $\A$ be an abelian category and let $V$ (resp. $Q$) be an object in $\A$ such that all coproducts (resp. products) of copies of $V$ (resp. Q) exist in $\A$.  If $Gen(V)\subseteq \Ker(\Ext^{1}_{\A}(V,?))$ (resp. $\Cogen(Q)\subseteq\Ker(\Ext^1_{\A}(?,Q))$), then the class $\Gen(V)$ (resp. $\Cogen(Q)$)  is a torsion (resp. torsionfree) class in $\A$.
 \end{lemma}
 \begin{proof}
 We will prove the assertion for $\Gen(V)$,  the one for $\Cogen(Q)$ following by duality. We will check that  the classes $\T:=\Gen(V)$ and $\F:=\T^\perp =V^\perp$ form a torsion pair, for which we just need to check 
 condition 2 of Definition \ref{def. torsion pair} since condition 1 is clearly satisfied. For any $A\in\text{Ob}(\A)$, we may consider the canonical map $\epsilon_A:V^{(\A(V,A))}\longrightarrow A$. This is the unique morphism such that $\epsilon_A\circ\iota_f=f$, where $\iota_f:V\longrightarrow V^{(\A(V,A))}$ is the $f$-th injection into the coproduct, for all $f\in\A(V,A)$. Its image is usually called the \emph{trace of $V$ in $A$} and is denoted by $\text{tr}_V(A)$. We then get an exact sequence $0\rightarrow \text{tr}_V(A)\hookrightarrow A\longrightarrow A/\text{tr}_V(A)\rightarrow 0$. We clearly have that $\text{tr}_V(A)\in\Gen(V)$. Moreover, we get an induced exact sequence of abelian groups $0\rightarrow\A(V,\text{tr}_V(A))\stackrel{\cong}{\longrightarrow}\A(V,A)\longrightarrow\A(V,A/\text{tr}_V(A))\longrightarrow\Ext_{\A}^1(V,\text{tr}_V(A))=0$. It then follows that $A/\text{tr}_V(A)\in V^\perp =\F$, so that condition 2 of Definition  \ref{def. torsion pair} is satisfied. 
 \end{proof}


We are ready to introduce some types of objects which have special importance in the study of the heart of a t-structure. They are generalizations of corresponding notions in module categories.

\begin{definition} \label{def.sub(co)generated} \normalfont
Let $\A$ be an abelian category and let $V$ be an object such that all coproducts of copies of $V$ exist in $\A$. We will say that an object $X$ is $V$-\emph{subgenerated} when it is isomorphic to a subobject of an object in $\Gen(V).$ The class of $V$-subgenerated objects will be denoted by $\overline{\Gen}(V).$ On the other hand, the class of objects on $\A$ which are isomorphic to direct summands of (resp. finite) coproducts of copies of $V$ will be denoted by $\Add(V)$ (resp. $\add(V)$). Dually, when $Q$ is an object such that all products of copies of $Q$ exist in $\A$, we  call  an object  $Q$-\emph{subcogenerated} when it is epimorphic image of an object in $\Cogen(Q)$. We denote by $\underline{\Cogen}(Q)$ and  $\Produc(Q)$ the subcategories consisting of $Q$-subcogenerated objects and objects isomorphic to direct summands of products of copies of $Q$, respectively.
\end{definition}

\begin{definition}\label{def. tilting object} \normalfont
Let $\A$ be an abelian category. An object $V$ (resp. Q) of $\A$ will be called \emph{quasi-tilting} (resp. \emph{quasi-cotilting}) when all coproducts (resp. products) of copies of $V$ (resp. $Q$) exist in $\A$ and $\Gen(V)=\overline{\Gen}(V) \cap \Ker(\Ext^{1}_{\A}(V,?)).$ (resp. $\Cogen(Q)=\underline{\Cogen}(Q)\cap\text{Ker}(\Ext^1_{\A}(?,Q))$). The corresponding torsion pair $\mathbf{t}=(\Gen(V),V^\perp)$ (resp. $\mathbf{t}=(_{}^\perp Q,\text{Cogen}(Q))$) (see Lemma \ref{Gen(V) torsion class}) is called the \emph{quasi-tilting (resp. quasi-cotilting) torsion pair} associated to $V$ (resp. $Q$). 

When, for such a $V$ (resp. $Q$), one has $\Gen(V)=\Ker(\Ext^1_{\A}(V,?)$ (resp. $\Cogen(Q)=\Ker(\Ext_{\A}^1(?,Q))$) and this class is cogenerating (resp. generating) in $\A$, we will say that $V$ (resp $Q$) is a \emph{1-tilting} (resp. \emph{1-cotilting}) object. The corresponding torsion pair is called the \emph{tilting} (resp. \emph{cotilting}) torsion pair associated to $V$ (resp. $Q$). 
\end{definition}
 
 The proof of the following goes as in module categories (see \cite[Proposition 2.1]{CDT}). 
 
\begin{corollary} \label{cor.CDT}
If $\A$ is an abelian category and $V$ (resp. $Q$) is a quasi-tilting (resp. quasi-cotilting) object of $\A$, then $\text{Gen}(V)=\text{Pres}(V)$ (resp. $\Cogen(Q)=\Copres(Q)$).
\end{corollary}

The natural question of when a quasi-tilting (resp. quasi-cotilting) torsion pair has a heart that is a category, i.e. has $Hom$ sets, has a clear answer:

\begin{corollary} \label{cor.tilting-projdim1}
Let $V$ (resp. $Q$) be a quasi-tilting (resp. quasi-cotilting) object of the abelian category $\A$, and let $\mathbf{t}=(\T,\F)$ the associated torsion pair in $\A$. The following assertions hold:

\begin{enumerate}
\item The heart $\Ht$ is a category (i.e. has $Hom$ sets) if, and only if, $\Ext_\A^1(V,F)$ (resp. $\Ext_\A^1(T,Q)$) is a set, for all $F\in\F$ (resp. $T\in\T$).
\item If $V$ (resp. $Q$) is a 1-tilting (resp. 1-cotilting) object, then  $\Ext_\A^2(V,?)=0$ (resp. $\Ext_\A^2(?,Q)=0$). One says that the projective (resp. injective) dimension of $V$ (resp. $Q$) is less or equal than $1$.
\item  If $V$ (resp. $Q$) is a 1-tilting (resp. 1-cotilting) object, then $\Ht$ is a category, i.e. it has $Hom$ sets.
\end{enumerate}
\end{corollary}
 \begin{proof}
 (1)  It is a direct consequence of Corollaries \ref{cor.heart-quasi-(co)tilting pair} and \ref{cor.CDT}.
 
 (2) We just do the proof for $V$, the one for $Q$ being dual. Let $0\rightarrow M\longrightarrow X\stackrel{f}{\longrightarrow} Y\longrightarrow V\rightarrow 0$ be an exact sequence in $\A$, representing an element $\epsilon\in\Ext_\A^2(V,M)$. Since $\T$ is a cogenerating class, we can fix a monomorphism $\mu:X\rightarrowtail T$, with $T\in\T$. By taking the pushout of $\mu$ and $f$ we immediately get an exact sequence $0\rightarrow M\longrightarrow T\stackrel{g}{\longrightarrow}T'\longrightarrow V\rightarrow 0$, where $T,T'\in\T$, which also represents $\epsilon$. But then $\epsilon =0$ since $\text{Im}(g)\in\T=\text{Ker}(\Ext_\A^1(V,?))$. 
 
   (3) Let $F\in\F$ be any object. Using the cogenerating condition of $\T$, we take an exact sequence $0\rightarrow F\longrightarrow T_0\longrightarrow T_1\rightarrow 0$, where $T_0,T_1\in\T$. We then get an exact sequence of (in principle big) abelian groups $$\A(V,T_1)\longrightarrow\Ext_A^1(V,F)\longrightarrow\Ext_A^1(V,T_0)=0. $$  It then follows that $\Ext_\A^1(V,F)$ is a set, which, by Corollary \ref{cor.heart-genuine-category}, implies that $\Ht$ has $Hom$ sets. 
 \end{proof}
 
 \subsection{When does the heart of a co-faithful (resp. faithful) torsion pair have a projective generator (resp. injective cogenerator)?}
 
 To answer the question of the title of this subsection we need a few preliminary results. 
 
 \begin{lemma} \label{lem-preservation of (co)products by class-homology}
Let $\mathcal{D}$ be a big triangulated category and $\tau =(\mathcal{U},\mathcal{W})$ be a nondegenerate t-structure in $\mathcal{D}$ whose heart $\A:=\mathcal{U}\cap\mathcal{W}$ is a category, i.e. it has $Hom$ sets. Let  $\mathbf{t}=(\mathcal{T},\mathcal{F})$  be a torsion pair in $\A$ such that $\Ext_{\A}^1(T,F)\cong\mathcal{D}(T,F[1])$ is a set
 (as opposite to a proper class), for all $T\in\T$ and all $F\in\F$ (see Proposition \ref{prop.HRS-tilt-big-triangcats}). The following assertions hold, where $\Ht$ denotes the heart of the tilted t-structure $\tau_\mathbf{t}$:

\begin{enumerate}
\item The functor $(H_{\tau}^0)_{| \Ht}:\Ht\longrightarrow\A$ is left adjoint of the functor $\A\longrightarrow\Ht$ taking $A\rightsquigarrow t(A)$, where $t:\A\longrightarrow\T$ is the torsion radical associated to $\mathbf{t}$. In particular  $(H_\tau^0)_{| \Ht}:\Ht\longrightarrow\A$ preserves all colimits that exist in $\Ht$.
\item The functor $(H_\tau^{-1})_{| \Ht}:\Ht\longrightarrow\A$ is right adjoint of the functor  $\A\longrightarrow\Ht$ taking $A\rightsquigarrow (1:t)(A)[1]$. In particular $(H_\tau^{-1})_{| \Ht}:\Ht\longrightarrow\A$ preserves all limits that exist in $\Ht$. 
\end{enumerate}
\end{lemma}
\begin{proof}
We just prove assertion 1 since assertion 2 follows by duality. By Proposition \ref{prop.HRS tilt}, given $M\in\Ht$, we have an exact sequence in $\Ht$  $$0\rightarrow H_\tau^{-1}(M)[1]\longrightarrow M\longrightarrow H_\tau^0(M)\rightarrow 0.$$
This sequence is precisely the one associated to the torsion pair $\bar{\mathbf{t}}=(\F[1],\T)$.
 Then the associated torsion radical $\bar{t}$ and coradical $(1:\bar{t})$ with respect to this torsion pair act on objects as $M\rightsquigarrow \bar{t}(M)=H_\tau^{-1}(M)[1]$ and $M\rightsquigarrow (1:\bar{t})(M)=H_\tau^0(M)$, respectively. We can then decompose $(H_\tau^0)_{| \Ht}:\Ht\longrightarrow\A$ as the composition $\Ht\stackrel{(1:\bar{t})\hspace{0,2 cm}}{\longrightarrow}\T\stackrel{\iota}{\hookrightarrow}\A$, where the right arrow is the inclusion functor. Each of the two functors in this composition has a right adjoint, which implies that $(H_\tau^0)_{| \Ht}:\Ht\longrightarrow\A$ has a right adjoint which is the composition $\A\stackrel{t}{\longrightarrow}\T\hookrightarrow\Ht$. 
\end{proof}

The importance of quasi-(co)tilting objects in the study of hearts of HRS t-structures stems from the following fact:

\begin{proposition} \label{prop.projective-generator-implies-qtilting}
Let $\A$ be an  abelian category and let $\mathbf{t}=(\mathcal{T},\mathcal{F})$ be a torsion pair in $\A$. If $\Ht$ is an AB3 (resp. AB3*) abelian category with a projective generator (resp. injective cogenerator) $P$ (resp. $E$), then $H^0(P)$ (resp. $H^{-1}(E)$) is a quasi-tilting (resp. quasi-cotilting) object and $\mathbf{t}$ is the associated quasi-tilting (resp. quasi-cotilting) torsion pair. 
\end{proposition}
\begin{proof}
The statement for the injective cogenerator is dual to the one for projective generator. We just do the last one. Let $P$ be as above and put $V:=H^{0}(P)$, and let $P^{(I)}$ denote the coproduct of $I$ copies of it in $\Ht$. We warn that it might not coincide with the corresponding coproduct in $\D(\A)$, if this one exists. By applying  Lemma \ref{lem-preservation of (co)products by class-homology} with $\tau=(\mathcal{D}^{\leq 0}(\mathcal{A}),\mathcal{D}^{\geq 0}(\mathcal{A}))$  the canonical t-structure, we have an isomorphism $H^0(P^{(I)})\cong H^0(P)^{(I)}=V^{(I)}$ in $\A$, so that all coproducts of copies of $V$ exist in $\A$. 

If $T\in\mathcal{T}$ is any object, then, due to the fact that $P$ is a projective generator of $\Ht$,  we have an exact sequence $P^{(I)}\longrightarrow P^{(J)}\longrightarrow T[0]\rightarrow 0$ in $\Ht$. By last paragraph, we get an exact sequence $H^0(P)^{(I)}\longrightarrow H^0(P)^{(J)}\longrightarrow H^0(T[0])=T\rightarrow 0$ in $\A$. We then get that $\mathcal{T}\subseteq\Pres(V)$, the converse inclusion being obvious. So we  have that $\mathcal{T}=\Gen(V)=\Pres(V)$. 

Moreover,  if we consider   the short exact sequence $0\rightarrow H^{-1}(P)[1]\longrightarrow P\longrightarrow V[0]\rightarrow 0$ in $\Ht$ and apply to it the long exact sequence of $\Ext_{\Ht}^*(?,T[0])$, we get  an exact sequence $$0=\Ht(H^{-1}(P)[1],T[0])\longrightarrow\Ext_{\Ht}^1(V[0],T[0])\longrightarrow\Ext_{\Ht}^1(P,T[0])=0,$$ from which we get that $\Ext_{\A}^1(V,T)\cong\Ext_{\Ht}^1(V[0],T[0])=0$, for all $T\in\mathcal{T}$. It then follows that $\mathcal{T}\subseteq\Ker (\Ext_{\A}^1(V,?))=:V^{\perp_1}$, and so $\mathcal{T}=\Gen(V)\subseteq\overline{\Gen}(V)\cap\Ker (\Ext_{\A}^1(V,?))$.  

For the reverse inclusion, given $M\in \overline{\Gen}(V) \cap V^{\perp_1}$, there exist $T_1,\, T_2\in \T$ and an exact sequence in $\A$ as follows:
\[
\xymatrix{
0 \ar[r] & M \ar[r] & T_1 \ar[r] & T_2 \ar[r] & 0.
}
\]
Since $\Pres(V)=\Gen(V)=\T$, we can take an epimorphism $q\colon V^{(\alpha)} \to T_2$ whose kernel belongs to $\T$. Consider the following pullback diagram 
\[
\xymatrix{
0 \ar[r] & M \ar[r] \ar@{=}[d] & Z \ar[r] \ar@{>>}[d] & V^{(\alpha)} \ar[r] \ar@{>>}[d]^{q} & 0 \\ 0  \ar[r] & M \ar[r] & T_1 \ar[r]  & T_2 \ar[r] \ar@{}[ul]|{\text{\tiny P.B.}} & 0
}
\]
Notice that $Z$ is an extension of $T_1$ and the kernel of $q$, so that $Z\in \T$. Taking into account that $M\in V^{\perp_1}=\text{Ker}(\text{Ext}_\A^1(V,?))=\text{Ker}(\text{Ext}_\A^1(V^{(I)},?))$, for each set $I\neq\emptyset$, we get that the first row in the diagram splits, so that $M\in \T$. 
\end{proof}

A first lesson of last proposition is that, in order to identify torsion pairs whose associated heart is a Grothendieck category, one can restrict to the quasi-cotilting ones. The proposition also helps in the following answer to the title of the subsection:

\begin{theorem} \label{thm.tilting theorem}
Let $\A$ be an abelian category and let $\mathbf{t}=(\mathcal{T},\mathcal{F})$ be a torsion pair in $\A$. The following assertions are equivalent:
\begin{enumerate}
\item $\mathbf{t}$ is a tilting torsion pair.
\item $\mathbf{t}$ is a co-faithful torsion pair whose heart $\Ht$ is an AB3 abelian category with a projective generator. 
\item $\Ht$ is an AB3 abelian category with a projective generator and $\bar{\mathbf{t}}=(\F[1],\T[0])$ is a faithful torsion pair in $\Ht$. 
\end{enumerate} 
In this case,  $V$ is a  1-tilting object such that $\T=\Gen(V)$ if,  and only if,  $V[0]$ is a projective generator of $\Ht$. Moreover, an object $P$ of $\Ht$ 
is a projective generator of this latter category if, and only if, it is isomorphic to $V[0]$ for some 1-tilting object $V$ of $\A$ such that $\T=\Gen (V)$. 
\end{theorem}
\begin{proof}
Note that in any of assertions (1)-(3) the class $\mathcal{T}$ is  cogenerating in $\mathcal{A}$. This is clear in assertions (1) and (2), and for assertion (3) it follows from Proposition \ref{prop.HRS1}. 

$(2)\Longleftrightarrow (3)$ is a consequence of this last mentioned proposition (= \cite[Proposition 3.2]{HRS}).

$(1)\Longrightarrow (2)$ Let $V$ be a 1-tilting object of $\A$ such that $\mathbf{t}=(\Gen(V),V^\perp)$. We start by proving that $V[0]$ is a projective object of $\Ht$, i.e. that $\text{Ext}_{\Ht}^1(V[0],M)=0$, for all $M\in\Ht$. But, taking into account the associated exact sequence $0\rightarrow H^{-1}(M)[1]\longrightarrow M\longrightarrow H^0(M)[0]\rightarrow 0$, the task reduces to the case when $M\in\T[0]\cup\F[1]$. If $M=T[0]$, with $T\in\T=\text{Ker}(\Ext_{\A}^1(V,?))$, then we have $\Ext_{\Ht}^1(V[0],T[0])\cong\Ext_\A^1(V,T)=0$. On the other hand,  if $F\in\F$ we have $\Ext_{\Ht}^1(V[0],F[1])\cong\Ext_\A^2(V,F)=0$ (see \cite[Remarque 3.1.17]{BBD} and Corollary \ref{cor.tilting-projdim1}). Note that what we have done with $V$ can be done with $V^{(I)}$, for any set $I\neq\emptyset$. That is, the argument also proves that $V^{(I)}[0]$ is a projective object of $\Ht$, for  for all sets $I$. 

Lemma \ref{lem.auxiliar} below says now that the stalk complex $V^{(I)}[0]$ is the coproduct in $\Ht$ of $I$ copies of $V[0]$. Moreover $\mathcal{T}[0]$ is a generating class in $\mathcal{H}_\mathbf{t}$ since $\bar{\mathbf{t}}=(\mathcal{F}[1],\mathcal{T}[0])$ is a faithful torsion pair due to Proposition \ref{prop.HRS1}. By the equality $\mathcal{T}=\text{Pres}_\mathcal{A}(V)$, we then get that $\mathcal{T}[0]\subseteq\text{Gen}_{\mathcal{H}_\mathbf{t}}(V[0])$, from which one immediately gets that $\mathcal{H}_\mathbf{t}=\text{Gen}_{\mathcal{H}_\mathbf{t}}(V[0])=\text{Pres}_{\mathcal{H}_\mathbf{t}}(V[0]).$
 Applying now Lemma \ref{lem.coproducts-Pres(X)}, we conclude that arbitrary coproducts exist in $\Ht$, so that this is an AB3 abelian category, with $V[0]$ as a projective generator.

\vspace*{0.3cm}

$(2)\Longleftrightarrow(3)\Longrightarrow (1)$ By Proposition \ref{prop.projective-generator-implies-qtilting} we know that $\mathbf{t}$ is a  quasi-tilting torsion pair. Let $V$ be a quasi-tilting object such that $\T=\Gen(V)$. Since $\mathbf{t}$ is co-faithful, i.e. $\T$ is a cogenerating  class in $\A$ we get that $\overline{\Gen}(V)=\A$, which then implies that $\overline{\Gen}(V)\cap\text{Ker}(\Ext_\A^1(V,?))=\text{Ker}(\Ext_\A^1(V,?))$, so that $\Gen(V)=\text{Ker}(\Ext_\A^1(V,?))$ and, hence, $V$ is a 1-tilting object. 

For the final statement, the proof of implication $(1)\Longrightarrow (2)$ shows that if $V$ is a 1-tilting object of $\A$ defining $\mathbf{t}$, then $V[0]$ is a projective generator of $\Ht$. It remains to prove that if $P$ is projective generator of $\Ht$ then $P\cong V[0]$ for such a 1-tilting object. By Proposition \ref{prop.projective-generator-implies-qtilting} we know that $V:=H^{0}(P)$ is a quasi-tilting object associated to $\mathbf{t}$, and by the argument in $(2)\Longrightarrow (1),$ it is even a 1-tilting object of $\A$. Then,  by  implication $(1)\Longrightarrow (2)$, we also know that $V[0]$ is  a projective generator of $\Ht$. It then follows that $P$ is a direct summand of the coproduct in $\Ht$ of $I$ copies of $V[0]$, for some set $I$. By Lemma \ref{lem.auxiliar} below, we then get that $P$ is a direct summand of $V^{(I)}[0]$, which implies that $H^{-1}(P)=0$ and hence that $P\cong V[0]$. 
\end{proof}

\begin{lemma} \label{lem.auxiliar}
Let $\A$ be an abelian category, let $V$ be a 1-tilting object, let $\mathbf{t}=(\Gen(V),V^\perp)$ be the associated torsion pair in $\A$ and let $\Ht$ be the heart of the associated HRS t-structure in $\D(\A)$. For each set $I$, the coproduct of $I$ copies of $V[0]$ exists in $\Ht$ and it is precisely the stalk complex $V^{(I)}[0]$. 
\end{lemma}
\begin{proof}
Let $\iota_j:V\longrightarrow V^{(I)}$ denote the $j$-th injection into the coproduct in $\A$, for each $j\in I$. For each $N\in\Ht$ we have an induced morphism $\gamma_{{}_{N}}:\Ht (V^{(I)}[0],N)\longrightarrow\Ht (V[0],N)^I$, which is the unique morphism of abelian groups such that the $j$-th projection $\pi_j:\Ht (V[0],N)^I\longrightarrow\Ht (V[0],N)$ satisfies $\pi_j\circ\gamma_{{}_{N}}=(\iota_j[0])^*(N)=\Ht(\iota_j[0],N):\Ht (V^{(I)}[0],N)\longrightarrow\Ht (V[0],N)$, for all $j\in I$. 

Our task reduces to prove that $\gamma_N$ is an isomorphism, for all $N\in\Ht$. To do that we consider the exact sequence $0\rightarrow H^{-1}(N)[1]\longrightarrow N\longrightarrow H^0(N)[0]\rightarrow 0$ in $\Ht$. Note that, by the first paragraph of the proof of implication $(1)\Longrightarrow (2)$ of last theorem, we know that $V^{(I)}[0]$ is projective in $\Ht$, for all sets $I$. This gives the following commutative diagram with exact rows:

\begin{small}
$$\xymatrix{0 \ar[r] & \Ht(V^{(I)}[0],H^{-1}(N)[1]) \ar[r] \ar[d]^{\gamma_{{}_{H^{-1}(N)[1]}}} & \Ht(V^{(I)}[0],N) \ar[r] \ar[d]^{\gamma_{{}_{N}}}& \Ht(V^{(I)}[0],H^{0}(N)[0]) \ar[r] \ar[d]^{\gamma_{{}_{H^{0}(N)[0]}}}& 0 \\ 0 \ar[r] &  \Ht(V[0],H^{-1}(N)[1])^{I} \ar[r] & \Ht(V[0],N)^{I} \ar[r] & \Ht(V[0],H^{0}(N)[0])^{I} \ar[r] & 0}$$
\end{small}

$\gamma_{{}_{H^{0}(N)[0]}}$ is clearly an isomorphism since it can be identified with the canonical map $\A(V^{(I)},H^0(N))\longrightarrow\A(V,H^0(N))^I$, which is an isomorphism by definition of the coproduct $V^{(I)}$ in $\A$. The task is further reduced to prove that $\gamma_{{}_{H^{-1}(N)[1]}}$ is an isomorphism. But this latter map gets identified with the canonical morphism $\gamma'_{{}_{F}}:Ext_\A^1(V^{(I)},F)\longrightarrow\Ext_\A^1(V,F)^I$, where $F:=H^{-1}(M)$. We just need to prove that $\gamma'_{{}_{F}}$ is an isomorphism, for all $F\in\F$. For this we use the cogenerating condition of $\T=\Gen(V)$ and, given $F\in\F$, we fix an exact sequence $0\rightarrow F\longrightarrow T\longrightarrow T'\rightarrow 0$, with $T,T'\in\T$. Bearing in mind that $\Ext_\A^1(V^{(J)},?)_{\arrowvert\T}=0$, for all sets $J$, we get the following commutative diagram with exact rows, where the two left vertical arrows are the canonical isomorphisms induced by definition of the coproduct $V^{(I)}$ in $\A$:

$$\xymatrix{ \A(V^{(I)},T) \ar[r] \ar[d] & \A(V^{(I)},T') \ar[r] \ar[d] & \Ext^{1}_{\A}(V^{(I)},F) \ar[r] \ar[d]^{\gamma'_{{}_{F}}} & 0 \\ \A(V,T)^{I} \ar[r] & \A(V,T')^{I} \ar[r] & \Ext^{1}_{\A}(V,F)^{I} \ar[r] & 0}$$

It follows that $\gamma'_{{}_{F}}$ is also an isomorphism as desired.
\end{proof}

Due to its importance, it is worth stating explicitly the dual of Theorem \ref{thm.tilting theorem}:

\begin{theorem} \label{thm.cotilting theorem}
Let $\A$ be an abelian category and let $\mathbf{t}=(\mathcal{T},\mathcal{F})$ be a torsion pair in $\A$. The following assertions are equivalent:
\begin{enumerate}
\item $\mathbf{t}$ is a cotilting torsion pair.
\item $\mathbf{t}$ is a faithful torsion pair whose heart $\Ht$ is an AB3* abelian category with an injective cogenerator. 
\item $\Ht$ is an AB3* abelian category with an injective cogenerator and $\bar{\mathbf{t}}=(\F[1],\T[0])$ is a co-faithful torsion pair in $\Ht$.  
\end{enumerate} 
In this case  $Q$ is a 1-cotilting object such that $\F=\Cogen(Q)$ if, and only if,  $Q[1]$ is an injective cogenerator of $\Ht$. Moreover, an object $E$ of $\Ht$ is an injective cogenertor of this category if, and only if, $E\cong Q[1]$ for some 1-cotilting object of $\A$ defining $\mathbf{t}$. 
\end{theorem}

We have now the following sort of reverse consequence:

\begin{corollary} \label{cor.reverse way}
Let $\A$ be an abelian category and  let $\mathbf{t}=(\mathcal{T},\mathcal{F})$ be a torsion pair in $\A$. The following assertions hold:
\begin{enumerate}
\item $\A$ is AB3 with a projective generator and $\mathbf{t}$ is a faithful torsion pair in $\A$ if, and only if, $\bar{\mathbf{t}}=(\F[1],\T[0])$ is a tilting torsion pair in $\Ht$. In such case,  $P$ is a projective generator of $\A$ if and only if $P[1]$ is a 1-tilting object of $\Ht$ such that $\F[1]=\Gen_{\Ht}(P[1])$.
\item $\A$ is AB3* with an injective cogenerator and $\mathbf{t}$ is a co-faithful torsion pair in $\A$ if, and only if, $\bar{\mathbf{t}}=(\F[1],\T[0])$ is a cotilting torsion pair in $\Ht$.  In such case,  $E$ is an injective cogenerator of $\A$ if and only if $E[0]$ is a 1-cotilting  object of $\Ht$ such that $\T[0]=\Cogen_{\Ht}(E[0])$.
\end{enumerate}
\end{corollary}
\begin{proof}
Obviously, each assertion is obtained from the other one by duality. We just prove assertion 2. By \cite[Proposition 3.2]{HRS} we know that $\T[0]$ is a generating class in $\Ht$, and, by \cite[Proposition I.3.4]{HRS} and using the terminology of that article, we have that $\Phi [(\Ht,\bar{\mathbf{t}})]$ is equivalent to $(\A,\mathbf{t})$, in fact it is equal to $(\A[1],\mathbf{t}[1])$. Moreover, by \cite[Theorem 3.3]{HRS} we even know that $\D^b(\A)$ and $\D^b(\Ht)$ are equivalent  triangulated categories. This allows us to apply Theorem \ref{thm.cotilting theorem},  replacing $\A$ by $\Ht$ and $\mathbf{t}$ by $\bar{\mathbf{t}}$ in that theorem, to conclude that $\bar{\mathbf{t}}$ is a cotilting torsion pair in $\Ht$. The last statement is also a consequence of Theorem \ref{thm.cotilting theorem}. 
\end{proof}

\subsection{Hearts that are module categories}
 
  In order to study  those hearts which are module categories,  we need the following concepts:

\begin{definition} \label{def.tilting-set-selfsmall} \normalfont
Let $\A$ be an abelian category and $\T_0$ be a set of objects such that arbitrary coproducts of objects of $\T_0$ exist in $\A$. We shall say that $\T_0$ is:

\begin{enumerate}
\item a \emph{1-tilting set} when $\coprod_{T\in\T_0}T$ is a 1-tilting object;
\item a \emph{self-small set} when, for each $T\in\T_0$ and each family $(T_\lambda)_{\lambda\in\Lambda}$ in $\T_0$, the canonical map $\coprod_{\lambda\in\Lambda}\A(T,T_\lambda )\longrightarrow\A(T,\coprod_{\lambda\in\Lambda}T_\lambda)$ is an isomorphism.
\item  a \emph{classical 1-tilting set} when it is 1-tilting and self-small.
\end{enumerate}

When $\T_0=\{T\}$ we say that $T$ is, respectively, a 1-tilting, a self-small and a classical 1-tilting object. 
\end{definition}

The following is the version that we will need of a theorem of Gabriel and Mitchell (see \cite[Corollary 3.6.4]{Po}):

\begin{proposition} \label{prop.Gabriel-Mitchell}
Let $\A$ be any category. The following assertions are equivalent:

\begin{enumerate}
\item $\A$ is equivalent to  $\text{Mod}-\mathcal{B}$ (resp. $\text{Mod}-R$), for some small pre-additive category $\mathcal{B}$ (resp. some ring $R$);
\item $\A$ is an AB3 abelian category that admits a self-small set of projective generators (resp. a self-small projective generator).
\end{enumerate}
\end{proposition}
\begin{proof}
The equivalence for $\text{Mod}-R$ is a particular case of the one for $\text{Mod}-\mathcal{B}$, for $\mathcal{B}$ a small pre-additive category, since a ring is the same as a pre-additive category with just one object. 
The classical version of Gabriel-Mitchell theorem states that assertion 1 holds if, and only if, $\A$ is AB3 and has a set  of small(=compact) projective generators (see, e.g., \cite[Corollary 3.6.4]{Po}). We just need to check that in any AB3 abelian category, if $\mathcal{P}_0$ is a self-small set of projective generators, then $\mathcal{P}_0$ consists of small objects. Indeed, let $(A_\lambda)_{\lambda\in\Lambda}$ be any family of objects in $\A$. For each $\lambda\in\Lambda$, we then have an exact sequence $\coprod_{P\in\mathcal{P}_0}P^{(I_{P,\lambda})}\stackrel{f_\lambda}{\longrightarrow}\coprod_{P\in\mathcal{P}_0}P^{(J_{P,\lambda})}\stackrel{p_{\lambda}}{\longrightarrow}A_\lambda\rightarrow 0$ in $\A$, for some sets $I_{P,\lambda}$ and $J_{P,\lambda}$. Due to right exactness of coproducts, we then get an exact sequence $$\coprod_{\lambda\in\Lambda}\coprod_{P\in\mathcal{P}_0}P^{(I_{P,\lambda})}\stackrel{\coprod f_\lambda}{\longrightarrow}\coprod_{\lambda\in\Lambda}\coprod_{P\in\mathcal{P}_0}P^{(J_{P,\lambda})}\stackrel{\coprod p_\lambda}{\longrightarrow}\coprod_{\lambda\in\Lambda}A_\lambda\rightarrow 0.$$

If now $P'\in\mathcal{P}_0$ is arbitrary and we apply $\A(P',?)$ to the last exact sequence, using the projectivity of $P'$ and the self-smallness of $\mathcal{P}_0$ we readily get that the canonical map $\coprod_{\lambda\in\Lambda}\A(P',A_\lambda)\longrightarrow\A(P',\coprod_{\lambda\in\Lambda}A_\lambda)$ is an isomorphism, so that $P'$ is small (=compact) in $\A$. 
\end{proof}

\begin{corollary} \label{cor.module-category}
Let $\A$ be an abelian category and let $\mathbf{t}=(\T,\F)$ be a torsion pair in $\A$. The following assertions are equivalent:

\begin{enumerate}
\item There is a classical 1-tilting set $\T_0$ (resp. a classical 1-tilting object $V$) such that $\T=\Gen(\T_0)$ (resp. $\T=\Gen(V)$).
\item $\mathbf{t}$ is a co-faithful torsion pair whose heart $\Ht$ is equivalent to the module category over a small pre-additive category (resp. over a ring). 
\item   $\Ht$ is equivalent to the module category over a small pre-additive category (resp. over a ring) and $\bar{\mathbf{t}}=(\F[1],\T[0])$ is a faithful  torsion pair in $\Ht$. 
\end{enumerate}
\end{corollary}
\begin{proof}

$(2)\Longleftrightarrow (3)$ is a consequence of Proposition \ref{prop.HRS1}.

$(1)\Longrightarrow (2)$ Since $V:=\coprod_{T\in\T_0}T$ is a 1-tilting object it follows from Theorem \ref{thm.tilting theorem} that $V[0]$ is a projective generator of $\Ht$, which in turns implies that $\T_0[0]=\{T[0]\text{: }T\in\T_0\}$ is a set of projective generators of $\Ht$. An easy adaptation of the proof of Lemma \ref{lem.auxiliar} shows that if $(T_\lambda )_{\lambda\in\Lambda}$ is a family in $\T_0$, then the coproduct of the $T_\lambda [0]$ in $\Ht$ exists and is the stalk complex $(\coprod_{\lambda\in\Lambda}T_\lambda) [0]$. Here the  $\mu$-th injection into the coproduct, for each $\mu\in\Lambda$, is the map $\iota_\mu [0]:T_\mu[0]\longrightarrow (\coprod_{\lambda\in\Lambda}T_\lambda)[0]$, where $\iota_\mu :T_\mu\longrightarrow\coprod_{\lambda\in\Lambda}T_\lambda$ is the $\mu$-th injection into the coproduct in $\mathcal{A}$. Given now $T\in\T_0$ arbitrary,  we have a sequence of isomorphisms:  $$\coprod_{\lambda\in\Lambda}\Ht (T[0],T_\lambda [0])\cong\coprod_{\lambda\in\Lambda}\A(T,T_\lambda )\stackrel{canonical}{\longrightarrow}\A(T,\coprod_{\lambda\in\Lambda}T_\lambda )\cong\Ht (T[0],(\coprod_{\lambda\in\Lambda}T_\lambda)[0]).$$ We claim that the composition of these isomorphisms, denoted by $\psi$ in the sequel,  is precisely the canonical morphism $$\coprod_{\lambda\in\Lambda}\Ht (T[0],T_\lambda [0])\longrightarrow\Ht (T[0],\coprod^{\Ht}_{\lambda\in\Lambda}(T_\lambda[0]))=\Ht (T[0],(\coprod_{\lambda\in\Lambda}T_\lambda)[0]). \hspace*{1cm} (\star )$$ To see this, for each $\mu\in\Lambda$ we let $u_\mu :\Ht (T[0],T_\mu[0])\longrightarrow\coprod_{\lambda\in\Lambda}\Ht (T[0],T_\lambda [0])$ the $\mu$-th injection into the coproduct in $\text{Ab}$.  We need to check that $\psi\circ u_\mu =(\iota_\mu[0])_*=\Ht (T[0],\iota_\mu [0])$, for all $\mu\in\Lambda$.  But this follows immediately from the equivalence of categories $\T\cong\T[0]$ and the fact that the composition $$ \mathcal{A}(T,T_\mu)\stackrel{\tilde{u}_\mu}{\longrightarrow}\coprod_{\lambda\in\Lambda}\A(T,T_\lambda)\stackrel{canonical}{\longrightarrow}\A(T,\coprod_{\lambda\in\Lambda}T_\lambda),$$ where $\tilde{u}_\mu$ is the canonical injection into the coproduct, is precisely the morphism $(\iota_\mu)_*=\A(T,\iota_\mu )$. 

  Therefore  $\T_0[0]$ is a self-small set of projective generators of $\Ht$ and, by Proposition \ref{prop.Gabriel-Mitchell}, we conclude that $\Ht\cong\text{Mod}-\mathcal{B}$, for some small pre-additive category $\mathcal{B}$. 

$(2)\Longleftrightarrow(3)\Longrightarrow (1)$ Due to the co-faithful condition on $\mathbf{t}$ and Proposition \ref{prop.HRS1}, the class $\T[0]$ is generating in $\Ht$. Hence any projective object of $\mathcal{H}_\mathbf{t}$ is in $\mathcal{T}[0]$. Then any self-small set of projective generators of $\mathcal{H}_\mathbf{t}$ is of the form $\mathcal{T}_0[0]$, for some set $\mathcal{T}_0\subset\mathcal{T}$. By Theorem \ref{thm.tilting theorem} and its proof, we get that $\mathbf{t}$ is the tilting torsion pair defined by the 1-tilting object $\hat{T}:=\coprod_{T\in\mathcal{T}_0}T$. It just remains to check that $\mathcal{T}_0$ is a self-small set. But this is a direct consequence of the self-smallness of $\mathcal{T}_0[0]$ since we have an equivalence of categories $\mathcal{T}\cong\mathcal{T}[0]$ and coproducts in $\mathcal{T}$ and $\mathcal{T}[0]$ are calculated as in $\mathcal{A}$ and $\mathcal{H}_\mathbf{t}$, respectively.
\end{proof} 

  \subsection{Bijections induced by the HRS tilting process}
  
  The previous results and the HRS tilting process give rise to a nice series of bijections that we gather in our next corollary.  We continue with the terminology of Definition \ref{def.HRS tilting process} and Proposition \ref{prop.HRS1}. They give  induced maps $(\mathbf{AB},\mathbf{tor}_{faithful})\stackrel{\longrightarrow}\longleftarrow(\mathbf{AB},\mathbf{tor}_{cofaithful})$. By Theorems \ref{thm.tilting theorem} and \ref{thm.cotilting theorem} and Corollary \ref{cor.reverse way}, we get the bijections in 1 and 2 of next corollary, and its bijections  $3$ and $4$   follow from Corollary \ref{cor.module-category}.

\begin{corollary} \label{cor.bijections-from-HRStilt}
Let  $\Phi:(\mathbf{AB},\mathbf{tor})\longrightarrow (\mathbf{AB},\mathbf{tor})$ be the map induced by the HRS tilting process (see Definition \ref{def.HRS tilting process}).
By restriction, $\Phi$ defines bijections, which are inverse of themselves (i.e. $(\Phi\circ\Phi)_{| \mathcal{C}}=1_\mathcal{C}$, for $\mathcal{C}$ any subclass in the list):

\begin{enumerate}
\item Between $(\mathbf{AB},\mathbf{tor}_{tilt})$ and $(\mathbf{AB3}_{proj},\mathbf{tor}_{faithful})$, where $\mathbf{tor}_{tilt}$ and $\mathbf{tor}_{faithful}$ denote the subclasses of $\mathbf{tor}$ consisting of the tilting and the faithful torsion pairs, respectively,   and $\mathbf{AB3}_{proj}$ denotes the class of AB3 abelian categories with a projective generator;
\item Between $(\mathbf{AB},\mathbf{tor}_{cotilt})$ and $(\mathbf{AB3}^*_{inj},\mathbf{tor}_{cofaithful})$, where $\mathbf{tor}_{cotilt}$  and $\mathbf{tor}_{cofaithful}$ denote the subclasses of $\mathbf{tor}$ consisting of the cotilting and the co-faithful torsion pairs, respectively, and  $\mathbf{AB3}^*_{inj}$ denotes the class of AB3* abelian categories with an injective cogenerator;
\item Between $(\mathbf{AB},\mathbf{tor}_{stilt-class})$ and $(\mathbf{Mod}_{paddt},\mathbf{tor}_{faithful})$, where $\mathbf{tor}_{stilt-class}$ denotes the subclass of $\mathbf{tor}_{tilt}$ consisting of those torsion pairs associated to a classical tilting set of objects and $\mathbf{Mod}_{paddt}$ is the class of abelian categories which are equivalent to module categories over small pre-additive categories. 
\item Between $(\mathbf{AB},\mathbf{tor}_{class-tilt})$ and $(\mathbf{Mod}_{ring},\mathbf{tor}_{faithful})$, where $\mathbf{tor}_{class-tilt}$ denotes the subclass of $\mathbf{tor}_{tilt}$ consisting of the torsion pairs associated to classical tilting objects and $\mathbf{Mod}_{ring}$ denotes the class of categories equivalent to module categories over rings. 
\end{enumerate}
\end{corollary}

Positselski and Stovicek have recently shown that complete and cocomplete abelian categories with an injective cogenerator and an $n$-tilting object correspond bijectively to complete and cocomplete abelian categories with a projective generator and an $n$-cotilting object \cite[Corollary 4.12]{Po-St2}. One can now recover the case $n=1$ in their result.


\begin{corollary}[Positselski-Stovicek] \label{cor.Po-St-tilting-cotilting-corresp}
The HRS tilting process gives a one-to-one correspondence between:
\begin{enumerate}
\item The pairs $(\A ,\mathbf{t})$ consisting of an AB3*  abelian category $\A$ with an injective cogenerator and a tilting torsion pair $\mathbf{t}$ in $\A$;
\item The pairs $(\mathcal{B},\bar{\mathbf{t}})$ consisting of an AB3  abelian category $\mathcal{B}$ with a projective generator and a cotilting torsion pair $\bar{\mathbf{t}}$ in $\mathcal{B}$
\end{enumerate}
Moreover the categories of assertion 1 are also AB4 and those of assertion 2 are also AB4*.  
\end{corollary}
\begin{proof}
We start by proving that any AB3* abelian category $\A$ with an injective cogenerator is AB4, that, together with its dual, will prove the last sentence of the corollary. Note that it is enough to prove that $\A$ is AB3 for it is well-known that  any AB3 abelian  category with an injective cogenerator is AB4   (see \cite[Corollary 3.2.9]{Po}). But proving the AB3 condition amounts  to prove that if $I$ is any set, viewed as a small category, then the constant diagram functor $\kappa :\A\longrightarrow\A^I$ has a left adjoint. This follows from Freyd's special adjoint theorem and its consequences (see \cite[Chapter 3, Exercises M, N]{Freyd}).

By Corollary \ref{cor.bijections-from-HRStilt} we have induced bijections $$\Phi :(\mathbf{AB},\mathbf{tor}_{tilt})\stackrel{\cong}{\longrightarrow} (\mathbf{AB3}_{proj},\mathbf{tor}_{faithful})$$ and $$\Phi :(\mathbf{AB}3^*_{inj},\mathbf{tor}_{cofaithful})\stackrel{\cong}{\longrightarrow}(\mathbf{AB},\mathbf{tor}_{cotilt}).$$ By restriction, we then get a bijection between the intersection of the domains and the intersection of the codomains. The intersection of the domains is precisely the class of pairs in 1 (note that the fact that $\T$ is cogenerating, equivalently that $\mathbf{t}\in\mathbf{tor}_{cofaithful}$, is automatic). Similarly, the intersection of codomains is precisely the class of pairs in 2. 
\end{proof}

\section{When is the heart of a torsion pair a Grothendieck category?}

\subsection{Initial results}

The work on the problem started with a series of papers \cite{CG}, \cite{CGM}, \cite{CMT}, \cite{MT}, which we now review in the terminology of this manuscript. Suppose that $\A$ is an abelian category with a classical tilting torsion $\mathbf{t}$ given by a 1-tilting object $V$. According to \cite[Theorem 3.3]{HRS} the realization functor gives an equivalence of triangulated categories $G:\mathcal{D}^b(\Ht)\stackrel{\cong}{\longrightarrow}\mathcal{D}^b(\A)$. On the other hand, by Corollary \ref{cor.module-category}, we know that $\Ht$ is a module category, actually via the equivalence of categories $\Ht (V[0],?):\Ht\stackrel{\cong}{\longrightarrow}\text{Mod}-R$, where $R=\text{End}_{\Ht}(V[0])\cong\text{End}_\A(V)$. Then we also have an equivalence of triangulated categories $\mathcal{D}^b(\text{Mod}-R)\stackrel{\cong}{\longrightarrow}\mathcal{D}^b(\Ht)$, and taking the composition, we get an induced equivalence of triangulated categories $\mathcal{D}^b(\text{Mod}-R)\stackrel{\cong}{\longrightarrow}\mathcal{D}^b(\A)$, taking $R$ to $V$. We can think of the inverse of this functor as a sort of right derived functor $\mathbb{R}H_V:\mathcal{D}(\A)\longrightarrow\mathcal{D}(\text{Mod}-R)$ of the canonical  functor $H_V:=\A(V,?):\A\longrightarrow\text{Mod}-R$. This last functor turns out to have a left adjoint $T_V:\text{Mod}-R\longrightarrow\A$ of which we can think as a sort of `tensor product by $V$'. We can then think of the equivalence $\mathbb{L}T_V:\D(\text{Mod}-R)\stackrel{\cong}{\longrightarrow}\D(\A)$ as a left derived functor of $T_V$. Note that for $A\in \A$ (resp. for $M \in \text{Mod}-R$),  $\mathbb{R}H_V(A)$ (resp. $\mathbb{L}T_V(M)$) is a complex and not just an $R$-module (resp. not just an object of $\A$). Concretely,  due to the fact that $\Ext_\A^2(V,?)=0$, one actually has that $\mathbb{R}H_V(A)$ has cohomology concentrated in degrees $0,1$, with $H^{0}(\mathbb{R}H_V(A))=\A(V,\A)$ and $H^1(\mathbb{R}H_V(A))=\Ext_\A^1(V,A)$, for all $A\in\A$. Dually, $\mathbb{L}T_V(M)$ has cohomology concentrated in degrees $-1,0$, with $H^0(\mathbb{L}T_V(M))=T_V(M)$ and $H^{-1}(\mathbb{L}T_V(M))=T'_V(M)$, where $T'_V:\text{Mod}-R\longrightarrow\A$ is the first left derived functor of $T_V$ in the classical sense, for all $R$-modules $M$.

Due to the definition of the torsion pair $\mathbf{t}$, one then  has $\mathbb{R}H_V(T)=\A(V,T)[0]$ and $\mathbb{R}H_V(F)=\Ext_\A^1(V,F)[-1]$. This implies that the equivalence $\mathbb{R}H_V:\D(A)\longrightarrow\D(\text{Mod}-R)$ induces equivalences of categories $$\F[1]\stackrel{\cong}{\longrightarrow}\mathcal{X}:=\{X\in\text{Mod}-R:\text{ }X\cong\Ext_\A^1(V,F)\text{, with }F\in\F\}$$

$$\T[0]\stackrel{\cong}{\longrightarrow}\mathcal{Y}:=\{Y\in\text{Mod}-R:\text{ }Y\cong\A(V,T)\text{, with }T\in\T\}.$$

Note that we then get induced equivalences of categories $H_V=\A(V,?):\T\stackrel{\cong}{\longrightarrow}\mathcal{Y}$ and $H'_V=\Ext_\A^1(V,?):\F\stackrel{\cong}{\longrightarrow}X$ whose quasi-inverses are necessarily $T_V:\mathcal{Y}\stackrel{\cong}{\longrightarrow}\T$ and $T'_V:\mathcal{X}\stackrel{\cong}{\longrightarrow}\F$. This essentially gives the proof of the following generalization of Brenner-Butler's theorem (see \cite{BB}),  due to Colpi and Fuller (see \cite[Theorem 3.2]{CF}):

\begin{theorem} \label{thm.Brenner-Butler}
Let $\A$ be an abelian category, let $V$ be a classical 1-tilting object in $\A$ and let $R=\text{End}_\A(V)$ the ring of endomorphisms of $V$. With the notation above, we have an equality of pairs $(\mathcal{X},\mathcal{Y})=(\text{Ker} (T_V),\text{Ker}(T'_V))$,  and this is a faitful torsion pair $\mathbf{t'}$ in $\text{Mod}-R$. Moreover, we have induced equivalences of categories $\xymatrix{\T \ar@<1ex>[r]^{H_V} & \ar@<1ex>[l]^{T_V}_{\sim} \mathcal{Y}}$ and $\xymatrix{\F \ar@<1ex>[r]^{H_V^{'}} & \ar@<1ex>[l]^{T^{'}_{V}}_{\sim} \mathcal{X} }$
\end{theorem}

In addition, by the paragraphs above, the torsion pair $\mathbf{t}'$ is sent to $\bar{\mathbf{t}}=(\F[1],\T[0])$ by the equivalence of categories $\text{Mod}-R\stackrel{\cong}{\longrightarrow}\Ht$. Then, using Proposition \ref{prop. Equiva A[1]Ht}, one gets the following initial result:

\begin{proposition}{[CGM, Corollary 2.4]}
$\A$ is equivalent to $\mathcal{H}_{\mathbf{t}'},$ where $\mathbf{t}'$ is as above.
\end{proposition}

On the other hand, we have a dual situation,  starting with $(\text{Mod-}R,\mathbf{s})$ a pair in $(\mathbf{Mod}_{ring},\mathbf{tor}_{faithful})$. It then follows that $R[1]$ is a classical 1-tilting object of $\mathcal{H}_{\mathbf{s}}$ (see Corollary \ref{cor.reverse way}) so that $\Phi[(\text{Mod-}R,\mathbf{s})]=(\mathcal{H}_{\mathbf{s}},\overline{\mathbf{s}}) \in (\mathbf{AB},\mathbf{tor}_{tilt})$. For this reason, the last result indicates that Question \ref{ques.HRS-tstructure} for faithful torsion pairs in  modules categories is equivalent to the question of  when  an abelian category $\A$ with a classical 1-tilting object is a Grothendieck category. This fact was exploited by Colpi, Gregorio and Mantese, who obtained the first partial answer to Question \ref{ques.HRS-tstructure}. 
\begin{theorem}{[CGM, Theorem 3.7]}
Let $(\A,\mathbf{t})\in (\mathbf{AB},\mathbf{tor}_{tilt})$ and we consider $\mathbf{t}'$ as above. Then, the following assertions are equivalent:
\begin{enumerate}

\item $\A\cong \mathcal{H}_{\mathbf{t}'}$ is a Grothendieck category;
\item for any direct system $(X_\lambda)_{\lambda}$ in $\mathcal{H}_{\mathbf{t}'}$ the canonical morphism $\varinjlim H_V(X_\lambda) \rightarrow H_V(\varinjlim_{\mathcal{H}_{\mathbf{t}'}} X_\lambda)$ is a monomorphism;

\item the functor $H_V$ preserves direct limits.

\end{enumerate}
If $\mathbf{t}'$ is of finite type, then the previous conditions are equivalent to the condition  that the functor $T_V \circ H_V$ preserve direct limits.
\end{theorem}

Already in \cite{CGM} the authors gave  necessary conditions for a faithful torsion pair in a module category  to have a heart which is  a Grothendieck category, a condition that was shown to be also sufficient in an unpublished paper by Colpi and Gregorio \cite{CG} (see \cite[Theorem 6.2]{Mat}).

\begin{theorem}{ \cite[Proposition 3.8]{CGM} and \cite[Theorem 1.3]{CG}}
Let $R$ be a ring,  let $\mathbf{t}=(\T,\F)$ be a faithful torsion pair in $\text{Mod-}R$ and let $\Ht$ be the heart of the associated Happel-Reiten-Smal\o \ t-structure in $\D(\text{Mod}-R)$. Then $\Ht$ is a Grothendieck category if, and only if, $\mathbf{t}=(\T,\F)$ is a cotilting torsion pair.
\end{theorem}

\subsection{The solution of the problem}

The solution to the problem was given by the authors in \cite{PS1} and $\cite{PS2}$. We realized that the hard part of the problem was to deal with the AB5 condition on $\Ht$. This  naturally led to a detailed study of direct limits in the heart. And, in order to understand those direct limits, it was a preliminary step to understand the behavior of the stalk complexes in the heart with respect to direct limits. 

\begin{proposition}{[PS1, Lemma 4.1 and Proposition 4.2]}
Let $\mathbf{t}=(\T,\F)$ be a torsion pair in the Grothendieck category $\G$. The following assertions  hold:
\begin{enumerate}
\item The functor $H^{0}:\Ht \rightarrow \G$ is right exact and preserve coproducts;

\item The functor $H^{-1}:\Ht \rightarrow \G$ is left exact and preserve coproducts;
 
\item For every $(M_\lambda)$ direct system in $\Ht$, the induced morphism $\varinjlim H^{k}(M_\lambda) \rightarrow H^{k}(\varinjlim_{\Ht} M_\lambda)$ is an epimorphism, for $k=-1$, and an isomorphism, for $k\neq -1.$

\item the pair $\overline{t}=(\F[1],\T[0])$ in $\Ht$ is of finite type.

\item For each   direct system $(F_\lambda)_{\lambda}$ in $\F$,   we have a canonical isomorphism $\varinjlim_{\mathcal{H}_\mathbf{t}}F_\lambda [1]\cong (1:t)(\varinjlim F_\lambda )[1]$.
\end{enumerate}
\end{proposition}

For instance, using the previous result and  Proposition \ref{prop. Equiva A[1]Ht}, we immediately get a necessary condition for  a positive answer in the case of a co-faithful torsion pair.

\begin{lemma}\label{lem. Short F=LimF}
Let $\G$ be a Grothendieck category and let $\mathbf{t}=(\T,\F)$ be a co-faithful torsion pair. If $\Ht$ is a Grothendieck category, then $\mathbf{t}$ is of finite type.
\end{lemma}
\begin{proof}
Suppose that $\Ht$ is a Grothendieck category. Since  $\G$ has enough injectives, from Proposition \ref{prop. Equiva A[1]Ht} we get that $\Ha_{\overline{\mathbf{t}}}$ is equivalent to $\G[1]$, via realization functor, where $\overline{\mathbf{t}}=(\F[1],\T[0])$ is the corresponding torsion pair in $\Ht$. Using assertion 4 of the previous proposition,  we deduce that $\overline{\overline{\mathbf{t}}}=(\T[1],\F[1])$ is a torsion pair of finite type in $\Ha_{\overline{\mathbf{t}}}\cong\mathcal{G}[1]$. That is, given a family $(F_{\lambda})_{\lambda}$ in $\F$, we have that $\varinjlim_{\mathcal{G}[1]}F_\lambda [1]\in\mathcal{F}[1]$, and this implies that $\varinjlim F_\lambda\in\mathcal{F}$ due to the canonical equivalence $\mathcal{G}\cong\mathcal{G}[1]$, which restricts to  $\mathcal{F}\cong\mathcal{F}[1]$.
\end{proof}

Another point of the strategy of the authors was to use the canonical cohomology functors $H^k:\mathcal{D}(\mathcal{G})\longrightarrow\mathcal{G}$ to approach the problem. In that way one gets sufficient conditions:

\begin{proposition}{[PS1, Proposition 3.4]}
Let $\G$ be a Grothendieck category and let $\sigma=(\U,\U^{\perp}[1])$ be a t-structure on $\D(\G)$. We denote its heart by $\Ha_{\sigma}$. If the classical cohomological functors $H^{k}:\Ha_{\sigma} \rightarrow \G$ preserve direct limits, for all integer $k$, then $\Ha_{\sigma}$ is an AB5 abelian category.
\end{proposition}

The following is now a natural question that remains open.

\begin{question}
Given a Grothendieck category $\G$ and a t-structure $\sigma=(\U,\U^{\perp}[1])$ in $\D(\G)$ such that its heart $\Ha_{\sigma}$ is an AB5 abelian category, do the classical cohomological functors $H^{m}:\Ha_{\sigma} \rightarrow \G$ preserve direct limits, for all $m\in \mathbb{Z}?$
\end{question}

 Recently, Chen, Han and Zhou have given necessary and sufficient conditions to have an equivalence  $\Ha_{\overline{\mathbf{t}}}$ is equivalent to $\G[1]$.

\begin{theorem}{[CHZ, Theorem A]} \label{thm.CHZ}
Let $\G$ be a Grothendieck category and let $\mathbf{t}=(\T,\F)$ be a torsion pair in $\G$. The following assertions are equivalent:
\begin{enumerate}
\item $\Ha_{\overline{\mathbf{t}}}$ is equivalent to $\G[1]$, via the realization functor;
\item Each object $X$ in $\G$ fits into an exact sequence
$$\xymatrix{0 \ar[r] & F^{0} \ar[r] & F^{1} \ar[r] & X \ar[r] & T^{0} \ar[r] & T^{1} \ar[r] & 0}$$
with $F^{i}\in \F$ and $T^{i} \in\mathcal{T}$, for $i=0,1$, and $\Ext^{3}_{\G}(T^{1},F^{0})=0$.
\end{enumerate}
\end{theorem}

The key point in the proof of Lemma \ref{lem. Short F=LimF} to guarantee that the class $\F$ is closed under direct limits is  the fact that $\mathcal{H}_{\overline{\mathbf{t}}}$ is in that situation equivalent to $\G[1]$ via the realization functor.  Then, keeping the same proof of that lemma, one immediately deduces from Theorem \ref{thm.CHZ}:

\begin{corollary}
Let $\G$ be a Grothendieck category and let $\mathbf{t}=(\T,\F)$ be a torsion pair in $\G$ that satisfies condition 2 of last theorem (e.g. any (co)faithful torsion pair). If $\Ht$ is a Grothendieck category, then $\mathbf{t}$ is of finite type. 
\end{corollary}

The general answer to Question \ref{ques.HRS-tstructure} was given by the authors. 

\begin{theorem}{[PS2, Theorem 1.2]}
Let $\G$ be a Grothendendieck and let $\mathbf{t}=(\T,\F)$ be a torsion pair in $\G$. Then, $\Ht$ is a Grothendieck category if, and only if, $\mathbf{t}$ is of finite type.
\end{theorem}

The main and harder part  was to prove that the finite type of $\mathbf{t}$ is a necessary condition. That is done in \cite[Theorem 4.8]{PS1} and the preliminary results leading to it. 
 We will give a new proof in Section \ref{Section new proof}. Conversely, for a torsion pair of finite type $\mathbf{t}=(\T,\F)$, the authors proved that $\T=\Pres(V)$, for some object $V$, and then, for a fixed generator $G$ of $\mathcal{G}$,  they showed in \cite[Proposition 4.7]{PS1} that the skeletally small subclass $\mathcal{N}$ of $\Ht$ consisting of those complexes $N$ such that $H^{-1}(N)$ is a subquotient of $G^m$  and $H^0(N)\cong V^n$, for some $m,n\in\mathbb{N}$,  is a class of generators of $\Ht$. Therefore, in order to answer Question \ref{ques.HRS-tstructure}, the only thing remaining was to prove that if $\mathbf{t}$ is of finite type then $\Ht$ is AB5. This was done in \cite[Theorem 1.2]{PS2}.

\subsection{A  side problem: When is the heart of a tilting torsion pair a Grothendieck category?}

Since tilting torsion pairs have hearts with a projective generator, it is good to know when that heart is a Grothendieck category, because such a heart would be very close to  a module category. The question of the title of this subsection has been recently  answered \cite[Corollary 2.5]{BHPST} for the case when the ambient Grothendieck category $\mathcal{G}$ is the module category over a ring. Recall that a \emph{pure exact sequence} in $\text{Mod}-R$ is a short exact sequence $0\rightarrow L\longrightarrow M\longrightarrow N\rightarrow 0$ that remains exact after applying the functor $?\otimes_RX$, for all left $R$-modules $X$. A module $P\in\text{Mod}-R$ is \emph{pure-projective} when the functor $\text{Hom}_R(P,?):\text{Mod-}R\longrightarrow\text{Ab}$ preserves exactness of pure exact sequences. 

\begin{theorem} \label{thm.6authors}{[BHPST, Corollary 2.5]}
Let $R$  be a ring, let $V$ be a 1-tilting (right) $R$-module and let $\mathbf{t}=(\Gen(V),V^\perp)$ be the associated torsion pair in $\text{Mod}-R$. The following assertions are equivalent:

\begin{enumerate}
\item $V$ is pure-projective.
\item $\mathbf{t}$ is of finite type (equivalently, the heart $\Ht$ is a Grothendieck category)
\end{enumerate}
\end{theorem}

It is well-known that a module is pure-projective if, and only if, it is a direct summand of a coproduct (=direct sum) of finitely presented modules. So when the heart of a tilting torsion pair  $\mathbf{t}=(\Gen(V),V^\perp)$ is a Grothendieck category, the projective generator of the heart $V[0]$ is determined by a set of `small objects'. Therefore the following question, first risen in \cite[Question 5.5]{PS1}, is apropos. We will call two modules $M$ and $N$ $Add$-equivalent when $\text{Add}(M)=\text{Add}(N)$.

\begin{question} \label{Saorin-question}
Let $V$ be a 1-tilting $R$-module whose associated torsion pair is  of finite type (equivalently, such that the heart $\Ht$ is a Grothendieck category). Is $V$ $Add$-equivalent to a classical 1-tilting module?. Equivalently, is the heart $\Ht$ equivalent to the module category over a ring?
\end{question}

It turns out that the answer to this question is negative in general, with counterexamples already existing when $R$ is a noetherian ring (see \cite[Section 4]{BHPST}). However, the following is true:

\begin{theorem} \label{thm.positive-answers-to-SaorinQuestion}
Let the ring $R$ satisfy one of the following conditions:

\begin{enumerate}
\item $R$ is a commutative ring;
\item $R$ is a Krull-Schmidt ring, i.e. every finitely presented (right) $R$-module is a direct sum of modules with local endomorphism ring (e.g. $R$ is right Artinian); 
\item every pure-projective (right) $R$-module is a coproduct of finitely presented modules.
\end{enumerate}
Then, a 1-tilting $R$-module is pure-projective if, and only if, it is $\text{Add}$-equivalent to a classical 1-tilting R-modules. Said in equivalent words, the heart of a tilting torsion pair in $\text{Mod}-R$ is a Grothendieck category if, and only if, it is equivalent to a module category over a ring. 
\end{theorem}
\begin{proof}
See \cite[Corollary 2.8 and Theorem 3.7]{BHPST}.
\end{proof}

\subsection{A new approach  using purity}\label{Section new proof}

Using now a recent result of Positselski and Stovicek \cite{Po-St} we can actually identify the cotilting torsion pairs in an abelian category for which the heart is  an AB5 abelian category. We need the following definition.

\begin{definition} \label{def.pure-injective} \normalfont
Let $\A$ be any additive category. We shall say that an object $Y$ of $\A$ is \emph{pure-injective} if the following two conditions hold:

\begin{enumerate}
\item The product of $Y^I$ exists in $\A$, for all sets $I$;
\item For each nonempty set $I$, there is a map $\phi :Y^I\longrightarrow Y$ such that $\phi\circ\iota_j=1_Y$, for all $j\in I$. Here $\iota_j:Y\longrightarrow Y^I$ is the unique morphism such that $\pi_i\circ\iota_j=\delta_{ij}1_Y$, with $\delta_{ij}$ the Kronecker symbol and $\pi_i:Y^I\longrightarrow Y$ the $i$-th projection. 
\end{enumerate}
We call the morphism $\iota_j$ in \emph{$j$-th injection into the product}.
\end{definition}

Note that,   if for $\A$ and $Y$ as in last definition,   also the coproduct $Y^{(I)}$ exists for all sets $I$,  then there is a canonical morphism $\kappa_Y :Y^{(I)}\longrightarrow Y^I$,  uniquely determined by the fact that the composition $Y\stackrel{\lambda_j}{\longrightarrow}Y^{(I)}\stackrel{\kappa_Y}{\longrightarrow}Y^I\stackrel{\pi_k}{\longrightarrow}Y$ equals $\delta_{jk}1_Y$, where $\lambda_j$ and $\pi_k$ are the $j$-th injection into the coproduct and $\pi_k$ is the $k$-th projection from the product, for all $j,k\in I$. We also have a \emph{summation map} $s_Y:Y^{(I)}\longrightarrow Y$, which is the only morphism such that $s_Y\circ\lambda_j=1_Y$, for all $j\in I$. We leave as an easy exercise for the reader to check that in this situation $Y$ is pure-injective if, and only if, this summation map $s_Y$ factors through $\kappa_Y$, for all sets $I$.  This completes the proof of the following result, which  is crucial for us:

\begin{lemma} (\cite[Theorem 3.3 (dual)]{Po-St})\label{lem.Positselski-Stovicek}
Let $\A$ be an AB3* abelian category with an injective cogenerator $E$ (whence $\A$ is also AB3 by the first paragraph of the proof of Corollary \ref{cor.Po-St-tilting-cotilting-corresp}). The following assertions are equivalent:

\begin{enumerate}
\item Direct limits are exact in $\A$, i.e. $\A$ is AB5.
\item The summation map $s_E:E^{(I)}\longrightarrow E$ factors through $\kappa_E:E^{(I)}\longrightarrow E^I$, for all sets $I$.
\item $E$ is a pure-injective object of $\A$.
\end{enumerate}
\end{lemma}

\begin{corollary} \label{cor.AB5-via-pureinjectivity}
Let $\A$ be an abelian category and $\mathbf{t}=(\mathcal{T},\F)$ be a torsion pair in $\A$. The following assertions are equivalent:

\begin{enumerate}
\item There is a pure-injective 1-cotilting object $Q$ of $\A$ such that $\F =\Cogen(Q)$;
\item $\mathbf{t}$ is a faithful torsion pair in $\A$ whose heart is an AB5 abelian category with an injective cogenerator;
\item The heart $\Ht$ is an AB5 abelian category with an injective cogenerator and $\bar{\mathbf{t}}$ is a co-faithful torsion pair in $\Ht$. 
\end{enumerate}
In particular, the HRS process gives a bijection `inverse of itself'  \newline $(\mathbf{AB},\mathbf{tor}_{cotilt-pinj})\stackrel{\cong}{\longrightarrow}(\mathbf{AB5}_{inj},\mathbf{tor}_{cofaithful})$, where:

 \begin{enumerate} 
\item[(a)] $(\mathbf{AB},\mathbf{tor}_{cotilt-pinj})$ consists of the pairs $(\A,\mathbf{t})$, where $\A$ is an abelian category and $\mathbf{t}=(\T,\F)$ is a torsion pair, with $\F=\Cogen(Q)$ for $Q$ a 1-coltilting pure-injective object.
\item[(b)] $(\mathbf{AB5}_{inj},\mathbf{tor}_{cofaithful})$ consists of the pairs $(\mathcal{B},\bar{\mathbf{t}})$, where $\mathcal{B}$ is an AB5 abelian category with an injective cogenerator and $\bar{\mathbf{t}}=(\mathcal{X},\mathcal{Y})$ is a co-faithful torsion pair in $\mathcal{B}$. 
\end{enumerate}
\end{corollary}
\begin{proof}
$(2)\Longleftrightarrow (3)$ It is a consequence of Proposition \ref{prop.HRS1}. 

$(1)\Longrightarrow (2)$ By dualizing the proof of Theorem \ref{thm.tilting theorem} and Lemma \ref{lem.auxiliar}, we know that $Q[1]$ is an injective cogenerator of $\Ht$ and that the stalk complex $Q^I[1]$ is the product of $I$ copies of $Q[1]$ in $\Ht$, for all sets $I$. If $\phi :Q^I\longrightarrow Q$ is a map such that $\phi\circ\iota_j=1_Y$, for all $j\in I$, with the notation of Definition \ref{def.pure-injective}, we then get that $\phi [1]:Q^I[1]\longrightarrow Q[1]$ satisfies that $\phi[1]\circ\iota_j[1]=1_{Y[1]}$, for all $j\in I$. But $\iota_j[1]:Y[1]\longrightarrow Y^I[1]$ is cleary in $j$-th injection into the product in $\Ht$. Therefore $Q[1]$ is pure-injective in $\Ht$. Since, by Theorem \ref{thm.cotilting theorem}, we know that $\Ht$ is AB3*, we can apply Lemma \ref{lem.Positselski-Stovicek}  to conclude that $\Ht$ is AB5.

$(2)\Longrightarrow (1)$ By Lemma \ref{lem.Positselski-Stovicek}  again, we know that $\Ht$ admits an injective cogenerator $E$ which is pure-injective. But since $\F[1]$ is a cogenerating class in $\Ht$ we necessarily have that $E=Q[1]$, for some $Q\in\F$. Now the dual of the proof of $(2)\Longleftrightarrow(3)\Longrightarrow (1)$ in Theorem \ref{thm.tilting theorem} shows that $Q$ is is a 1-cotilting object of $\A$ such that $\F=\Cogen(Q)$ and $\F$ is a generating class in $\A$. The argument in the proof of $(1)\Longrightarrow (2)$ proves that $Q$ is pure-injective in $\A$ if and only if $Q[1]=E$ is pure-injective in $\Ht$, something that we know by hypothesis. 
\end{proof}	

 Recall that if $\A$ is an abelian category, then an \emph{abelian exact subcategory} is a full subcategory $\mathcal{B}$ that is abelian and such that the inclusion functor $\mathcal{B}\hookrightarrow\A$ is exact. This is equivalent to say that $\mathcal{B}$ is closed under taking finite coproducts, kernels and cokernels in $\A$.  We are now in a position to re-prove the hard part of \cite[Theorem 1.2]{PS2}, that is the proof of \cite[Theorem 4.8]{PS1}, by using recent results in the literature.

\begin{theorem} \label{thm.Parra-Saorin-Theorem}
Let $\mathcal{G}$ be a Grothendieck category and let $\mathbf{t}=(\T,\F)$ be a torsion pair in $\mathcal{G}$. If the heart $\Ht$ of the associated Happel-Reiten-Smalo t-structure in $\mathcal{D}(\mathcal{G})$ is a Grothendieck category, then $\mathbf{t}$ is of finite type, i.e. $\F$ is closed under taking direct limits in $\mathcal{G}$. 
\end{theorem}
\begin{proof}
By Proposition \ref{prop.projective-generator-implies-qtilting} we know that $\F=\Cogen(Q)=\Copres(Q)$, for some quasi-cotilting object $Q$. Consider now the subcategory $\underline{\F}=\underline{\Cogen}(Q)$ of $\G$ (see Definition \ref{def.sub(co)generated}). This subcategory is clearly closed under taking subobjects, quotients and coproducts, so that it is an abelian exact subcategory where colimits are calculated as in $\mathcal{G}$. In particular $\underline{\F}$ is an AB5 abelian category. Moreover, if $X$ is a generator of $\mathcal{G}$ one readily gets that $(1:t)(X)$ is a generator of $\underline{\F}$, so that this subcategory is actually a Grothendieck category. 

Note also that the inclusion functor $\iota :\underline{\F}\hookrightarrow\mathcal{G}$ has a right adjoint $\rho:\mathcal{G}\longrightarrow\underline{\F}$. The action on objects is given by $\rho (M)=\text{tr}_{\F}(M)$, where $\text{tr}_\F(M)$ is the trace of $\F$ in $M$, i.e. the subobject sum of all subobjects of $M$ which are in $\underline{\F}$.  We leave to the reader the easy verification that $(\iota ,\rho)$ is an adjoint pair. We can then derive these functors. Due to the exactness of $\iota$, the left derived of $\iota$, $\mathbb{L}\iota =\iota$ `is' $\iota$ itself, i.e. it just takes a complex $X^\bullet\in\mathcal{D}(\underline{\F})$ to the same complex viewed as an object of $\mathcal{D}(\mathcal{G})$. The right derived $\mathbb{R}\rho :\mathcal{D}(\mathcal{G})\longrightarrow\mathcal{D}(\underline{\F})$ is defined in the usual way, namely, it is the composition $\mathcal{D}(\mathcal{G})\stackrel{\mathbf{i}}{\longrightarrow}\mathcal{K}(\mathcal{G})\stackrel{\rho}{\longrightarrow}\mathcal{K}(\underline{\mathcal{F}})\stackrel{q}{\longrightarrow}\mathcal{D}(\underline{F})$, where $\mathbf{i}$ is the homotopically injective resolution functor, (abusing of notation) $\rho$ is the obvious functor induced at the level of homotopy categories, and $q$ is the canonical localization functor. Then, by classical properties of derived functors and derived categories, we get that $(\iota,\mathbb{R}\rho)$ is an adjoint pair of triangulated functors. 

Consider now the restricted torsion pair $\mathbf{t}'=(\T\cap\underline{\F},\F)$ in $\underline{\F}$. By \cite[Proposition 3.2]{PS1}, we know that its heart $\mathcal{H}_{\mathbf{t}'}$ is an AB3 abelian category. Moreover, the triangulated functor $\iota :\mathcal{D}(\underline{\mathcal{F}})\longrightarrow\mathcal{D}(\mathcal{G})$ clearly satisfies that $\iota (\mathcal{H}_{\mathbf{t}'})\subseteq\Ht$. We therefore get and induced functor $\iota :\mathcal{H}_{\mathbf{t}'}\longrightarrow\Ht$, which is necessarily exact since short exact sequences in hearts are the triangles in the ambient triangulated category with their three vertices in that heart. We claim that the composition of functors $\rho': \Ht\hookrightarrow\mathcal{D}(\mathcal{G})\stackrel{\mathbf{R}\rho}{\longrightarrow}\mathcal{D}(\underline{\F})\stackrel{H_{\mathbf{t}'}^0}{\longrightarrow}\mathcal{H}_{\mathbf{t}'}$ is right adjoint of $\iota :\mathcal{H}_{\mathbf{t}'}\longrightarrow\Ht$. Let $X\in\mathcal{H}_{\mathbf{t}'}$ and $M\in\Ht$ be arbitrary objects. Note that we can identify $M$ with a complex $ \cdots\rightarrow 0\rightarrow E^{-1}\rightarrow E^0\rightarrow E^1\rightarrow \cdots$ of injective objects of $\mathcal{G}$ concentrated in degrees $\geq -1$. Then $\mathbb{R}\rho (M)$ is the complex $\cdots \rightarrow  0\rightarrow \rho (E^{-1})\rightarrow\rho (E^0)\rightarrow\rho (E^1)\rightarrow \cdots$. As a right adjoint, the functor $\rho:\mathcal{G}\longrightarrow\underline{\F}$ is left exact, and this implies that $H^{-1}(\mathbb{R}\rho (M))\cong \rho (H^{-1}(M))\cong H^{-1}(M)$ since $H^{-1}(M)\in\mathcal{F}$. This implies that $\mathbb{R}\rho (M)\in\mathcal{W}_{\mathbf{t}'}$, where $\mathcal{W}_{\mathbf{t}'}$ is the coaisle of the HRS t-structure in $\mathcal{D}(\underline{\F})$ associated to $\mathbf{t}'$. Remember that the restriction of $H_{\mathbf{t}'}^0$ to $\mathcal{W}_{\mathbf{t}'}$ is right adjoint of the inclusion functor $\mathcal{H}_{\mathbf{t}'}\hookrightarrow\mathcal{W}_{\mathbf{t}'}$ (see \cite[Lemma 3.1(2)]{PS1}). We then have a sequence of isomorphisms, natural on both variables:  

$$\begin{array}{llllll}
\mathcal{H}_{\mathbf{t}'}(X,\rho'(M))&=& \mathcal{H}_{\mathbf{t}'}(X,(H_{\mathbf{t}'}^0\circ\mathbb{R}\rho)(M))  & \cong& \mathcal{W}_{\mathbf{t}'}(X,\mathbb{R}\rho (M)) \\
& & & = &\mathcal{D}(\underline{\F})(X,\mathbb{R}\rho (M)) \\
&  &  & \cong& \mathcal{D}(\mathcal{G})(\iota (X),M) \\ 
&&& \cong &\Ht (\iota (X),M),
\end{array}$$
which implies that $(\iota,\rho')$ is an adjoint pair. 

We then get that the exact functor $\iota:\mathcal{H}_{\mathbf{t}'}\longrightarrow\Ht$ preserves direct limits. Moreover, it reflects zero objects since $\iota (X)=0$ means that $X$ is  acyclic, viewed as a complex of objects of $\mathcal{G}$, which is the same as being acyclic when viewed as a complex of objects in $\underline{\F}$. Consider now a direct system $(0\rightarrow L_i\stackrel{u_i}{\rightarrowtail}M_i)_{i\in I}$ of monomorphisms in $\mathcal{H}_{\mathbf{t}'}$ and, putting $u:=\varinjlim (u_i)$, consider the exact sequence in $\mathcal{H}_{\mathbf{t}'}$ \begin{center}$0\rightarrow\text{Ker}_{\mathcal{H}_{\mathbf{t}'}}(u)\longrightarrow\varinjlim_{\mathcal{H}_{\mathbf{t}'}}L_i\stackrel{u}{\longrightarrow}\varinjlim_{\mathcal{H}_{\mathbf{t}'}}M_i.  $ \end{center} By exactness and preservation of direct limits by $\iota:\mathcal{H}_{\mathbf{t}'}\longrightarrow\Ht$, we get an exact sequence in $\Ht$ \begin{center}$0\rightarrow\iota (\text{Ker}(u))\longrightarrow\varinjlim_{\Ht}\iota (L_i)\stackrel{\varinjlim\iota (u_i)}{\longrightarrow}\varinjlim_{\Ht}\iota (M_i). $ \end{center} Since the $\iota (u_i)$ are monomorphism and $\Ht$ is AB5 we get that $\varinjlim\iota (u_i)$ is a monomorphism, so that $\iota (\text{Ker}(u))=0$. The fact that $\iota$ reflects zero objects then implies that $u$ is a monomorphism. Therefore $\mathcal{H}_{\mathbf{t}'}$ is also AB5.  

On the other hand, we claim that $Q$ is a 1-cotilting object of $\underline{\F}$ and that $\mathbf{t}'$ is its associated torsion pair in $\underline{\F}$.  Indeed, since we know that $\mathcal{F}=\Cogen(Q)=\underline{\mathcal{F}}\cap\text{Ker}(\text{Ext}_\G^1(?,Q))$, it is enough to check that $\underline{\mathcal{F}}\cap\text{Ker}(\text{Ext}_\G^1(?,Q))=\Ker(\Ext_{\underline{\mathcal{F}}}^1(?,Q))$. The inclusion ``$\subseteq$'' is clear. For the converse, let $M\in\Ker(\Ext_{\underline{\mathcal{F}}}^1(?,Q))$ and fix two exact sequences 
 \[
 0\rightarrow F'\stackrel{u}{\longrightarrow}F\longrightarrow M\rightarrow 0 \qquad\text{and}\qquad 0\rightarrow F'\stackrel{v}{\longrightarrow}Q^I\longrightarrow F''\rightarrow 0
 \] 
 with $F,\, F',\, F''\in\mathcal{F}$ and $I$ some set (where for the second exact sequence we used that $\mathcal{F}=\Copres(Q)$). Taking the pushout of $u$ and $v$, we obtain the following commutative diagram with exact rows and columns:
 \[
 \xymatrix@R=14pt{
 &0\ar[d]&0\ar[d]\\
 0\ar[r]& F'\ar@{}[dr]|{\text{P.O.}}\ar[r]^u\ar[d]_{v}&F\ar[r]\ar[d]&M\ar[r]\ar@{=}[d]&0\\
 0\ar[r]&Q^I\ar[d]\ar[r]&X\ar[d]\ar[r]&M\ar[r]&0\\
 &F''\ar[d]\ar@{=}[r]&F''\ar[d]\\
 &0&0
 }
 \] 
We then obtain that $X\in\mathcal{F}$ (as it is an extension of $F$ and $F''\in\F$), so that $Q^I,\, X,$ and $M$ all belong in $\underline{\mathcal{F}}$. By the choice of $M$, the second row of the diagram splits, so that $M\in\mathcal{F}$ since it is isomorphic to a direct summand of $X$.

It now follows from Theorem \ref{thm.cotilting theorem}, Corollary \ref{cor.AB5-via-pureinjectivity} and  the proof of the latter that $Q$ is pure-injective in $\underline{\F}$. But then $\mathcal{F}$ is closed under taking direct limits in the Grothendieck category $\underline{\F}$ (see \cite[Theorem 3.9]{Cou-St}), which is equivalent to say that is is closed under taking direct limits in $\mathcal{G}$. That is, $\mathbf{t}=(\T,\F)$ is a torsion pair of finite type, as desired.

\end{proof}
  
\section{Beyond the HRS case:  Some recent results} \label{sec.beyondHRS}

After Question \ref{ques.HRS-tstructure} was solved, 
as said in the introduction, it is Question \ref{ques.smashing-t-structure} the one that has deserved more attention. So far, the work was mainly concentrated on the case when the t-structure $(\U,\W$) is compactly generated. Then one can even assume that the ambient triangulated category $\D$ is compactly generated. This is due to the fact that $\mathcal{L}:=\text{Loc} _\mathcal{D}(\mathcal{U})$, the smallest triangulated subcategory of $\D$ containing $\U$ and closed under taking arbitrary coproducts, is compactly generated and the restricted t-structure $\tau'=(\mathcal{U},\mathcal{U}\cap\mathcal{L})$ has the same heart as $\tau$. 

In the compactly generated case, partial answers to the question were obtained by using different techniques, such as functor categories (\cite{AMV}, \cite{Bo}), stable $\infty$-categories \cite{Lurie}  and the theory of derivators \cite{SSV}, see also \cite{PS3} and \cite{Bazz} for particular cases. These investigations suggest that for all compactly generated t-structures appearing in nature the heart is a Grothendieck category. The concluding result in this vein has been recently obtained independently in \cite{Bo2} and \cite{SS}:


\begin{theorem}[\cite{Bo2}  and \cite{SS}]
Let $\mathcal{D}$ be a triangulated category with coproducts and $\tau =(\mathcal{U},\mathcal{W})$ be a compactly generated t-structure in $\D$. Then the heart $\mathcal{H}_\tau=\mathcal{U}\cap\mathcal{W}$ is a Grothendieck category. 
\end{theorem}

In the development via derivators of \cite{SSV}, the new concept of \emph{homotopically smashing t-structure} (with respect to a strong stable derivator) was introduced. We refer to that reference for the definition and to \cite{Groth} for all the terminology concerning derivators.  All compactly generated t-structures that appear as the base of a strong stable derivator are homotopically smashing. The latter t-structures are always smashing, but the converse is not true.  For instance the HRS t-structure is always smashing, but it is homotopically smashing exactly when the torsion pair is of finite type (see \cite[Proposition 6.1]{SSV}). The following is a combination of \cite[Theorems B and C]{SSV}, and we refer to that reference for all unexplained terminology appearing in the statement:

 \begin{theorem} \label{thm.SSV}
 Let $\mathbb{D}:\text{Cat}^{op}\longrightarrow\text{CAT}$ be a strong stable derivator, with base $\D:=\mathbb{D}(\mathbf{1})$, and let $\tau =(\mathcal{U},\mathcal{W})$ be a t-structure in $\D$ that is homotopically smashing with respect to $\mathbb{D}$, then the heart $\mathcal{H}_\tau$ is an AB5 abelian category. When, in addition, $\mathbb{D}$ is the derivator associated to the homotopy category of a stable combinatorial model structure and $\tau$ is generated by a set, that heart is a Grothendieck category.
 \end{theorem}

Soon after \cite{SSV} appeared,  Rosanna Laking \cite{L} proved the following result:

\begin{theorem}
Let $\mathcal{D}$ be a compactly generated triangulated category that is the base of a strong stable derivator $\mathbb{D}$, and let $\tau =(\mathcal{U},\mathcal{W})$ be a left nondegenerate t-structure in $\D$. The following assertions are equivalent:

\begin{enumerate}
\item $\tau$ is homotopically smashing with respect to $\mathbb{D}$.
\item $\tau$ is smashing and the heart $\mathcal{H}_\tau$ of $\tau$ is a Grothendieck category.
\end{enumerate}
\end{theorem}

The last two results suggest the following open question:

\begin{question} 
Let $\D$ be a well-generated triangulated category (e.g. a compactly generated one) that is the base of a strong stable derivator $\mathbb{D}$. Are the following two conditions equivalent for a t-structure $\tau =(\mathcal{U},\mathcal{W})$ in $\D$?

\begin{enumerate}
\item $\tau$ is homotopically smashing with respect to $\mathbb{D}$.
\item $\tau$ is smashing and the heart $\mathcal{H}_\tau$ of $\tau$ is a Grothendieck category.
\end{enumerate}
\end{question}

In order to get (a partial version of) this question in a derivator-free way, a hint comes from \cite[Theorem 4.6]{L} (see also \cite[Theorem 4.7]{LV}), where the author proves that 'homotopically smashing' and 'definable' are synonymous terms for the co-aisle of left nondegenerated t-structures, when the ambient triangulated category is the compactly generated base of a strong stable derivator (see \cite{L} for the definition of definable subcategory of a compactly generated triangulated category). This suggests the following question:

\begin{question}
Let $\D$ be a compactly generated triangulated category.  Are the following two conditions equivalent for a t-structure $\tau =(\mathcal{U},\mathcal{W})$ in $\D$?:
\begin{enumerate}
\item $\mathcal{W}$ is definable.
\item $\tau$ is smashing and its heart $\mathcal{H}_\tau$  is a Grothendieck category.
\end{enumerate}
\end{question}

\medskip
Manuel Saor\'in -- \texttt{msaorinc@um.es}\\
{Departamento de Matem\'{a}ticas,
Universidad de Murcia,  Aptdo.\,4021,
30100 Espinardo, Murcia,
SPAIN}

\medskip
Carlos E. Parra -- \texttt{carlos.parra@uach.cl}\\
{Instituto de Ciencias F\'isicas y Matem\'aticas, Edificio Emilio Pugin, Campus Isla Teja, Universidad Austral de Chile, 5090000 Valdivia, CHILE}

\end{document}